\documentclass[a4paper,12pt,dvipdfmx]{article}
\usepackage{amsmath,amssymb,amsthm}
\usepackage{mathtools}
\numberwithin{equation}{section}
\usepackage[margin=20mm]{geometry}

\newtheorem{thm}{Theorem}[section]
\newtheorem{prop}[thm]{Proposition}

\newtheorem{rem}[thm]{Remark}
\newtheorem{lem}[thm]{Lemma}

\newtheorem{assum}[thm]{Assumption}

\DeclareMathOperator*{\esssup}{ess\,sup}

\def\supp{{\rm supp}}

\makeatletter

\begin{document}

\title{Volume growth, big jump, and essential spectrum \\
for regular Dirichlet forms\thanks{This work was 
supported by JSPS KAKENHI Grant Numbers JP22K18675, JP23K25773.}}
\author{Yuichi Shiozawa\thanks{Department of Mathematical Sciences, Faculty of Science and Engineering, Doshisha University, 
Kyotanabe, Kyoto, 610-0394, Japan; \texttt{yshiozaw@mail.doshisha.ac.jp}}}

\maketitle

\begin{abstract}
We establish an upper bound of the bottom of the essential spectrum 
for the generator associated with a regular Dirichlet form 
in terms of the rates of the volume growth/decay and big jump. 
Using this bound, we discuss how the bottom of the essential spectrum is affected 
by the volume growth and coefficient growth.  
\end{abstract}

\section{Introduction}
We are concerned with the spectral properties of the $L^2$-generator 
associated with a regular Dirichlet form. 
We first establish an upper bound of the bottom of the essential spectrum 
in terms of the volume growth/decay rate of the underlying space and 
the big jump rate for the Dirichlet form (Theorems \ref{thm:ess} and \ref{thm:ess-f}). 
We then apply this result to a class of non-local Dirichlet forms. 
These applications suggest the validity of our upper bound in terms of the positivity of the bottom of the essential spectrum.

It is well known that an $L^2$-Markovian semigroup is noncompact 
if and only if the essential spectrum of the associated generator is nonempty.  
Hence a noncompact Markovian semigroup is in fact an infinite dimensional object, 
and the domain of the corresponding Dirichlet form is large in this sense.  
In particular, we can regard the bottom of the essential spectrum 
as a characteristic quantity of noncompactness. 

On the other hand, 
the volume growth of the underlying measure is one of the benchmarks 
for the global properties of Markovian semigroups such as conservativeness and recurrence. 
In fact, there are several criteria in terms of the volume growth rate 
for the validity of these properties 
(see, e.g., \cite{F14-1, G99, GHM12, HKS20, MUW12, S15, St94}). 

In connection with noncompactness of Markovian semigroups, 
it is natural to relate the bottom of the essential spectrum 
to the volume growth rate. 
For the Laplace-Beltrami operator on a noncompact and complete Riemannian manifold,  
Brooks \cite{B81, B84} established a precise upper bound of the bottom of the spectrum 
in terms of the volume growth/decay rate (see also \cite{H01, K09} for refinements). 
This result was generalized to strongly local regular Dirichlet forms 
via the notion of intrinsic metrics (\cite{N98}),  and to weighted manifolds (\cite{R17}). 
Folz \cite{F14} further extended the result of Brooks \cite{B81} 
to Dirichlet forms on weighted graphs via the notion of adapted metrics.  
By using the notion of intrinsic metrics in the sense of  Frank-Lenz-Wingert \cite{FLW14}, 
Haeseler-Keller-Wojciechowski \cite{HKW13} also extended the result of Brooks \cite{B81} 
to regular Dirichlet forms without the killing term, with applications to weighted graphs. 
See \cite[Section 13.2 and Note (p.~524)]{KLW21} for the exposition and related references 
on these results.

The previous works mentioned above 
concern regular Dirichlet forms with strong locality or graph structure. 
Even though the formulation of Haeseler-Keller-Wojciechowski \cite{HKW13} covers 
general regular Dirichlet forms with infinite volume, 
the use of intrinsic metrics may force the small and big jump parts
to have finite moments of the common order (see \cite[Sections 14.3 and 14.4]{FLW14}). 
On the other hand, 
for a non-local Dirichlet form, 
we know necessary and sufficient conditions for compactness and transience of the Markovian semigroup  
in terms of the growth rates of the coefficients for the small and big jump parts (\cite{S23, SW23}).  
Therefore, it is natural to separate the small and big jumps  
for getting the upper bound of the bottom of the spectrum.  
Our objective in this paper is to pursue this approach so that 
we extend the previous works 
to regular Dirichlet forms having non-locality and no graph structure.   

We accomplish our objective by following the approach of \cite{S15, S16}. 
More precisely, we divide the non-local part of a Dirichlet form 
into the relatively small and big jump parts  
by introducing the adapted length and jump height function (see Assumption \ref{assum:length}). 
We can then apply the argument in the previous works (\cite{B81, B84, F14, HKW13, N98, R17}) to the relatively small jump part.  
We also extract the big jump rate from the relatively big jump part. 
Our results (Theorems \ref{thm:ess} and \ref{thm:ess-f}) and their proofs are presented in Section \ref{sect:thm-proof}.

In Section \ref{sect:volume}, we focus on the relation  
between the volume growth and the essential spectrum. 
As will be mentioned in Remark \ref{rem:hyperbolic}, 
our results provide a nontrivial upper bound of the bottom of the essential spectrum 
for a non-local Dirichlet form  
under the exponential volume growth condition with respect to the original distance. 

In Section \ref{sect:coeff}, we apply our results to 
two kinds of non-local Dirichlet forms with unbounded coefficients. 
We here take into consideration 
the coefficient growth rate into the adapted length and jump height function. 
When the state space is Euclidean space, 
we know necessary and sufficient conditions for noncompactness of associated Markovian semigroups 
in terms of the coefficient growth rate (\cite{CW14, HW25, M21, SW23, W19, WZ21}).  
Under the setting as \cite{M21, SW23, W19}, 
we further get a lower bound of the bottom of the essential spectrum 
by making use of Persson's formula (\cite{LS19}) 
and the Lyapunov method (\cite{S23, SW23}). 
Combining this bound with Theorem \ref{thm:ess}, we have a quantitative characterization of noncompactness 
in terms of the positivity of the bottom of the essential spectrum.

In Section \ref{sect:ou-type}, 
motivated by \cite{WW15}, 
we study the bottom of the spectrum for a non-local operator related to 
the Ornstein-Uhlenbeck type. 
We will see that, even if the volume is finite,  
the big jump rate may contribute to the positivity of the bottom of the essential spectrum. 
See Remark \ref{rem:ou-type} for details. 

\section{Preliminaries}
In this section, we set up notation and terminology on the Dirichlet form theory by following \cite{FOT11}. 
We also introduce assumptions which will be needed throughout this paper.

Let $E$ be a locally compact separable metric space, 
and $m$ a positive Radon measure on $E$ with full support. 
Let $C(E)$ be the totality of continuous functions on $E$, 
and $C_0(E)$ the totality of functions in $C(E)$ with compact support. 
Let $({\cal E},{\cal F})$ be a regular Dirichlet form on $L^2(E;m)$. 
By definition, ${\cal F}\cap C_0(E)$ is dense in $C_0(E)$ with respect to the uniform norm, 
and also dense in ${\cal F}$ with respect to the norm 
$\|u\|_{\cal E}=({\cal E}(u,u)+\|u\|_{L^2(E;m)}^2)^{1/2} \ (u\in {\cal F})$.

Let ${\cal B}(E)$ denote the totality of Borel measurable subsets of $E$. 
We impose the next assumption on the Beurling-Deny expression of $({\cal E},{\cal F})$ 
(\cite[Theorem 3.2.1]{FOT11}).
\begin{assum}\label{assum:BD}\rm  
For any $u\in {\cal F}\cap C_0(E)$, 
\[
{\cal E}(u,u)={\cal E}^{(c)}(u,u)+\iint_{E\times E\setminus {\rm diag}}(u(x)-u(y))^2\, J(x,{\rm d}y)m({\rm d}x),
\]
where
\begin{itemize}
\item ${\rm diag}$ is the diagonal set in $E\times E$, that is, ${\rm diag}=\{(x,y)\in E\times E \mid x=y\}$.
\item $({\cal E}^{(c)},{\cal F}\cap C_0(E))$ is a symmetric form with the strongly local property 
(see \cite[p.120]{FOT11} for definition).
\item $J(x,{\rm d}y)$ is a positive measurable kernel on $(E,{\cal B}(E))$ such that 
the measure $J({\rm d}x,{\rm d}y)=J(x,{\rm d}y)m({\rm d}x)$ is symmetric, that is, 
\[
J(A\times B)=J(B\times A), \quad A,B\in {\cal B}(E).
\]
\end{itemize}
\end{assum}

The symmetric form ${\cal E}^{(c)}$ can be extended to ${\cal F}$ (\cite[p.~125]{FOT11}). 
Moreover, for any $u\in {\cal F}$, there exists a unique finite measure $\mu_{\langle u \rangle}^{(c)}$ on $E$ 
such that 
\[
{\cal E}^{(c)}(u,u)=\frac{1}{2}\mu_{\langle u \rangle}^{(c)}(E)
\]
(\cite[Lemma 3.2.3]{FOT11}). 
We call $\mu_{\langle u \rangle}^{(c)}$ the local part of the energy measure of $u\in {\cal F}$.

We say that a function $u$ on $E$ is locally in ${\cal F}$ ($u\in {\cal F}_{{\rm loc}}$ in notation) 
if for any relatively compact open set $G$ in $E$, 
there exists a function $u_G\in {\cal F}$ such that $u=u_G$, $m$-a.e.\ on $G$.
For any $u\in {\cal F}_{{\rm loc}}$, we can well define the measure 
which is consistent  with the local part of the energy measure (\cite[p.~130]{FOT11}).
We use the notation $\mu_{\langle u \rangle}^{(c)}$ also for this measure. 
When $\mu_{\langle u \rangle}^{(c)}$ is absolutely continuous with respect to the measure $m$, 
we write $\Gamma^{(c)}(u)$ for the Radon-Nikodym derivative, 
that is, $\mu_{\langle u \rangle}^{(c)}({\rm d}x)=\Gamma^{(c)}(u)(x)\,m({\rm d}x)$.

We further make the following assumption on the existence of a family of length functions adapted to $({\cal E},{\cal F})$. 

\begin{assum}\label{assum:length}\rm 
There exist  families $\{\rho_r\}_{r>0}\subset {\cal F}_{{\rm loc}}\cap C(E)$ and 
$\{F_r\}_{r>0}\subset C(E\times E)$ such that 
$F_r$ is positive and pointwisely increasing in $r>0$, and the following conditions are satisfied:
\begin{enumerate}
\item[(i)] 
For each $r>0$, the measure $\mu_{\langle \rho_r \rangle}^{(c)}$ is absolutely continuous 
with respect to $m$ and $\esssup_{x\in E}\{\Gamma^{(c)}(\rho_r)(x)\}<\infty$.

\item[(ii)] For any $r>0$, 
\[
\esssup_{x\in E}\int_{0<d(x,y)\le F_r(x,y)}(\rho_r(x)-\rho_r(y))^2\,J(x,{\rm d}y)<\infty
\]
and 
\[
\esssup_{x\in E}\int_{d(x,y)>F_r(x,y)}J(x,{\rm d}y)<\infty.
\]
\end{enumerate} 
Here $\esssup$ is the essential supremum with respect to $m$. 
\end{assum}

Let $\{\rho_r\}_{r>0}\subset {\cal F}_{{\rm loc}}\cap C(E)$ and $\{F_r\}_{r>0}\subset C(E\times E)$ 
satisfy  Assumption \ref{assum:length}. 
For $r>0$, we define 
\[
M_1(r)=\esssup_{x\in E}\left\{\Gamma^{(c)}(\rho_r)(x)\right\}+\esssup_{x\in E}\int_{0<d(x,y)\le F_r(x,y)}(\rho_r(x)-\rho_r(y))^2\,J(x,{\rm d}y)
\]
and 
\[
M_2(r)=\esssup_{x\in E}\int_{d(x,y)>F_r(x,y)}J(x,{\rm d}y). 
\]
For $r>0$ and $R>0$, we also define 
\[
K_{\rho_r}(R)=\left\{x\in E \mid \rho_r(x)\le R\right\}.
\]
We further impose the topological assumption on $K_{\rho_r}(R)$, 
which implies that for any $r>0$, $K_{\rho_r}(R)\nearrow E$ as $R\rightarrow\infty$.

\begin{assum}\label{assum:compact}\rm 
For any $r>0$ and $R>0$, $K_{\rho_r}(R)$ is compact. 
Moreover, for any $r>0$ and for any compact set $K\subset E$, there exists $R>0$ such that $K\subset K_{\rho_r}(R)$. 
\end{assum}

\section{Upper bound of the bottom of the essential spectrum}\label{sect:thm-proof}
Let $({\cal E},{\cal F})$ be a regular Dirichlet form on $L^2(E;m)$. 
Then there exists a unique nonpositive self-adjoint operator $(L,{\cal D}(L))$ on $L^2(E;m)$ such that 
${\cal D}(L)\subset {\cal F}$ and 
\begin{equation}\label{eq:generator}
{\cal E}(u,v)=(-Lu,v), \quad u\in {\cal D}(L), \, v\in {\cal F}.
\end{equation}
Let $\sigma_{{\rm ess}}(-L)$ be the essential spectrum of $-L$, 
and let $\lambda_e=\inf\sigma_{{\rm ess}}(-L)$. 
To show the upper bound of $\lambda_e$, we rely on the following Persson theorem:

\begin{prop}{\rm (see \cite[Proposition 2.1]{HKW13} or \cite[Theorem E.2]{KLW21})}\label{prop:persson}
If there exists a sequence $\{f_n\}\subset {\cal F}$ with $\|f_n\|_{L^2(E;m)}=1$ 
converging weakly to $0$ in $L^2(E;m)$, then 
\[
\lambda_e\le \liminf_{n\rightarrow\infty}{\cal E}(f_n,f_n).
\]
\end{prop}

Throughout this section, 
we impose Assumptions \ref{assum:BD}, \ref{assum:length} and \ref{assum:compact} on $({\cal E},{\cal F})$. 
We will establish upper bounds of $\lambda_e$ for the infinite volume and finite volume cases, respectively, 
by using Proposition \ref{prop:persson}.

\subsection{Infinite volume}\label{subsect:infinite}
In this subsection, we assume that $m(E)=\infty$. 
For $r>0$, let 
\[
\mu_r=\liminf_{R\rightarrow\infty}\frac{1}{R}\log m(K_{\rho_r}(R)).
\]

\begin{thm}\label{thm:ess}
Assume that $\mu_r<\infty$ for some $r>0$. Then 
\[
\lambda_e \le \inf_{r>0} \left( \frac{\mu_r^2}{4} M_1(r)+2M_2(r) \right).
\]
\end{thm}

By comparison with the previous works \cite{B81, F14, HKW13, N98, R17}, 
Theorem \ref{thm:ess} includes not only the volume growth rate, 
but also the big jump rate expressed as $M_2(r)$. 
A key point in our argument is to take into consideration  the degree of farness of points in $E$ 
by introducing the function $F_r(x,y)$. 
We can then adapt the approach of the previous works \cite{B81, F14, HKW13, N98, R17} 
to the relatively small jump part.

In what follows, we assume that $\mu_r<\infty$ for some $r>0$. 
Then there exists a positive sequence $\{R_n\}$ such that $R_n \rightarrow \infty$ as $n \rightarrow \infty$ and
\[
\mu_r=\lim_{n\rightarrow\infty}\frac{1}{R_n}\log m(K_{\rho_r}(R_n)).
\]
We fix such a sequence $\{R_n\}$. 

For any $\alpha>\mu_r/2$, let $w_n \ (n\ge 1)$ be a continuous function on $[0,\infty)$ such that 
\[
w_n(t)=
\begin{dcases}
\alpha R_n/2, & 0\le t \le R_n/2, \\
\alpha (R_n-t), & R_n/2 < t \le R_n, \\
0, & t>R_n. 
\end{dcases}
\]
Define $h_n^{(r)}(x)=w_n(\rho_r(x))$ and $f_n^{(r)}(x)=e^{h_n^{(r)}(x)}-1$. 
We first prove that 
we can take the sequence $\{f_n^{(r)}/\|f_n^{(r)}\|_{L^2(E;m)}\}$ 
as $\{f_n\}$ in  Proposition \ref{prop:persson}.

\begin{lem}\label{lem:local}
\begin{enumerate}
\item[{\rm (i)}] For any $n\ge 1$, $h_n^{(r)}$ and $f_n^{(r)}$ belong to ${\cal F} \cap C_0(E)$.
\item[{\rm (ii)}] The sequence $\{ f_n^{(r)}/\|f_n^{(r)}\|_{L^2(E;m)} \}$ is weakly convergent to $0$ in $L^2(E;m)$. 
In particular, 
\begin{equation}\label{eq:persson-1}
\lambda_e \le \liminf_{n\rightarrow\infty}\frac{{\cal E}(f_n^{(r)},f_n^{(r)})}{\|f_n^{(r)}\|_{L^2(E;m)}^2}.
\end{equation}
\end{enumerate}
\end{lem}

\begin{proof}
We first prove (i). 
Since $h_n^{(r)}$ and $f_n^{(r)}$ belong to $C_0(E)$ by definition and Assumption \ref{assum:compact}, 
it is sufficient to show that $h_n^{(r)}$ and $f_n^{(r)}$ belong to ${\cal F}$. 

Since $\rho_r\in {\cal F}_{{\rm loc}}$, 
we can  follow the proof of \cite[Lemma 3.1]{S15} to see that $h_n^{(r)}\in {\cal F}$. 
We also know that $|e^t-e^s|\le e^{s\vee t}|t-s|$ for any $s,t\in {\mathbb R}$. 
Then 
\begin{equation*}
\begin{split}
|f_n^{(r)}(x)-f_n^{(r)}(y)|=|e^{h_n^{(r)}(x)}-e^{h_n^{(r)}(y)}|
&\le e^{h_n^{(r)}(x)\vee h_n^{(r)}(y)}|h_n^{(r)}(x)-h_n^{(r)}(y)|\\
&\le e^{\alpha R_n/2}|h_n^{(r)}(x)-h_n^{(r)}(y)|
\end{split}
\end{equation*}
and 
\[
|f_n^{(r)}(x)|=|e^{h_n^{(r)}(x)}-1|
\le e^{h_n^{(r)}(x)}|h_n^{(r)}(x)| \le e^{\alpha R_n/2} |h_n^{(r)}(x)|.
\]
Since $h_n^{(r)}\in {\cal F} \cap C_0(E)$, 
we have $f_n^{(r)}\in {\cal F} \cap C_0(E)$ by \cite[p.\ 5]{FOT11}.

We next prove (ii). 
Take any $u\in L^2(E;m)$ with $u\ne 0$. Then for any $R>0$,  
\[
\int_E u f_n^{(r)} \,{\rm d}m=\int_E u f_n^{(r)} {\bf 1}_{K_{\rho_r}(R)} \, {\rm d}m+\int_E u f_n^{(r)} {\bf 1}_{K_{\rho_r}(R)^c} \, {\rm d}m.
\]
By the Schwarz inequality, 
\[
\left| \int_E u f_n^{(r)} {\bf 1}_{K_{\rho_r}(R)} \, {\rm d}m \right|
\le \left( \int_E u^2\,{\rm d}m\right)^{1/2} \left(\int_E (f_n^{(r)})^2{\bf 1}_{K_{\rho_r}(R)}\,{\rm d}m\right)^{1/2}
\]
and 
\[
\left| \int_E u f_n^{(r)} {\bf 1}_{K_{\rho_r}(R)^c} \, {\rm d}m \right|
\le \left(\int_E u^2{\bf 1}_{K_{\rho_r}(R)^c}\,{\rm d}m\right)^{1/2} \left( \int_E (f_n^{(r)})^2 \,{\rm d}m \right)^{1/2}.
\]

By the definition of $f_n^{(r)}$, 
\begin{equation}\label{eq:fn-div}
\int_E (f_n^{(r)})^2\,{\rm d}m
\ge \int_{\rho_r(x) \le R_n/2} (f_n^{(r)})^2\,{\rm d}m
= (e^{\alpha R_n/2}-1)^2 m(K_{\rho_r}(R_n/2))\rightarrow\infty, \quad n\rightarrow\infty.
\end{equation}
Let $\varepsilon>0$. 
Then by Assumption \ref{assum:compact}, there exists $R_*>0$ such that 
\[
\left(\int_E u^2{\bf 1}_{K_{\rho_r}(R_*)^c}\,{\rm d}m\right)^{1/2}<\frac{\varepsilon}{2}.
\]
Note that for any $n\in {\mathbb N}$ with $R_n/2\ge R_*$, 
\[
\int_E (f_n^{(r)})^2{\bf 1}_{K_{\rho_r}(R_*)}\,{\rm d}m\le (e^{\alpha R_n}/2-1)^2 m(K_{\rho_r}(R_*)).
\]
Then, as \eqref{eq:fn-div} holds and $m(E)=\infty$ by assumption,
there exists $N_0\in {\mathbb N}$ such that for any $n\ge N_0$,
\[
\left( \frac{\int_E (f_n^{(r)})^2{\bf 1}_{K_{\rho_r}(R_*)}\,{\rm d}m}{\| f_n^{(r)} \|_{L^2(E;m)}^2} \right)^{1/2} 
= \left( \frac{\int_E (f_n^{(r)})^2{\bf 1}_{K_{\rho_r}(R_*)}\,{\rm d}m}{\int_E (f_n^{(r)})^2 \,{\rm d}m} \right)^{1/2} 
< \frac{\varepsilon}{2\|u\|_{L^2(E;m)}}.
\]
Therefore, for any $n\ge N_0$, 
\begin{equation*}
\begin{split}
\left| \int_E u f_n^{(r)}  \,{\rm d}m \right|
&= \left| \int_E u f_n^{(r)} {\bf 1}_{K_{\rho_r}(R_*)} \, {\rm d}m + \int_E u f_n^{(r)} {\bf 1}_{K_{\rho_r}(R_*)^c} \, {\rm d}m \right| \\
&\le  \left| \int_E u f_n^{(r)} {\bf 1}_{K_{\rho_r}(R_*)} \, {\rm d}m \right| + \left| \int_E u f_n^{(r)} {\bf 1}_{K_{\rho_r}(R_*)^c} \, {\rm d}m \right| \\
&\le  \|u\|_{L^2(E;m)} \left(\int_E (f_n^{(r)})^2{\bf 1}_{K_{\rho_r}(R_*)}\,{\rm d}m\right)^{1/2} 
+ \left(\int_E u^2{\bf 1}_{K_{\rho_r}(R_*)^c}\,{\rm d}m\right)^{1/2} \| f_n^{(r)} \|_{L^2(E;m)} \\
&< \varepsilon \| f_n^{(r)} \|_{L^2(E;m)}.
\end{split}
\end{equation*}
This completes the proof of the first assertion of (ii). 
Proposition \ref{prop:persson} further yields \eqref{eq:persson-1}. 
\end{proof}

We next discuss the upper bound of ${\cal E}(f_n^{(r)},f_n^{(r)})$. 
Let $g_n^{(r)}=(f_n^{(r)}+2){\bf 1}_{K_{\rho_r}(R_n)}$ and 
\begin{equation}\label{eq:varphi}
\varphi(t)=\frac{(1-e^{\alpha t})^2}{1+e^{2\alpha t}}, \quad t\in {\mathbb R}.
\end{equation}
Then $\varphi$ is even and nonnegative, and strictly increasing for $t\ge 0$ such that 
\[
\varphi(t)=1-\frac{2e^{\alpha t}}{1+e^{2\alpha t}} \le \frac{\alpha^2t^2}{2}, \quad t\in {\mathbb R}.
\]

\begin{lem}\label{lem:upper-bound}
\begin{enumerate}
\item[{\rm (i)}] For any $n\ge 1$, 
\begin{equation}\label{eq:f-bound}
(f_n^{(r)}(x)-f_n^{(r)}(y))^2
\le \varphi(|\rho_r(x)-\rho_r(y)|) (g_n^{(r)}(x)^2+g_n^{(r)}(y)^2), \quad x,y\in E.
\end{equation}
In particular, 
\[
(f_n^{(r)}(x)-f_n^{(r)}(y))^2
\le \frac{\alpha^2}{2}(\rho_r(x)-\rho_r(y))^2 (g_n^{(r)}(x)^2+g_n^{(r)}(y)^2), \quad x,y\in E.
\]
\item[{\rm (ii)}] 
For any $n\ge 1$, 
\[
{\cal E}(f_n^{(r)},f_n^{(r)})
\le \alpha^2 M_1(r) \int_E (g_n^{(r)})^2\,{\rm d}m+ 2M_2(r) \int_E (f_n^{(r)})^2\,{\rm d}m.
\]
\end{enumerate}
\end{lem}

\begin{proof}
We omit the proof of (i) because it is the same with that of \cite[Lemma 2.5]{HKW13}. 
Let us prove (ii).
For $u\in {\cal F}_{{\rm loc}} \cap C(E)$, we define 
\[
{\cal E}^{(r),1}(u,u)=\iint_{0<d(x,y)\le F_r(x,y)}(u(x)-u(y))^2\,J(x,{\rm d}y)m({\rm d}x)
\]
and 
\[
{\cal E}^{(r),2}(u,u)=\iint_{d(x,y)>F_r(x,y)}(u(x)-u(y))^2\,J(x,{\rm d}y)m({\rm d}x)
\]
so that 
\begin{equation}\label{eq:decomp}
{\cal E}(u,u)={\cal E}^{(c)}(u,u)+{\cal E}^{(r),1}(u,u)+{\cal E}^{(r),2}(u,u).
\end{equation}

Since $w_n(t)=0\vee \{\alpha(R_n-t)\}\wedge (\alpha R_n/2)$ and $R_n/2-\rho_r\in {\cal F}_{{\rm loc}}$, 
we have by the chain rule and the contraction property of $\Gamma^{(c)}$ (\cite[p.\ 190]{St94}), 
\begin{equation*}
\begin{split}
{\cal E}^{(c)}(f_n^{(r)},f_n^{(r)})
&=\int_E \Gamma^{(c)}(f_n^{(r)})\,{\rm d}m
=\int_E  e^{2w_n(\rho_r)}\Gamma^{(c)}(w_n(\rho_r))\,{\rm d}m\\
&\le \alpha^2 \int_E (g_n^{(r)})^2 \Gamma^{(c)}(\rho_r)\,{\rm d}m
\le \alpha^2 \sup_{z\in E}\left\{\Gamma^{(c)}(\rho_r(z))\right\} \int_E (g_n^{(r)})^2\,{\rm d}m.
\end{split}
\end{equation*}
By (i), $F_r(x,y)=F_r(y,x)$ and $J(x,{\rm d}y)m({\rm d}x)=J(y,{\rm d}x)m({\rm d}y)$, 
we also have 
\begin{equation*}
\begin{split}
{\cal E}^{(r),1}(f_n^{(r)},f_n^{(r)})
&=\iint_{0<d(x,y)\le F_r(x,y)}(f_n^{(r)}(x)-f_n^{(r)}(y))^2\,J(x,{\rm d}y)m({\rm d}x)\\
&\le \frac{\alpha^2}{2}\iint_{0<d(x,y)\le F_r(x,y)}(\rho_r(x)-\rho_r(y))^2(g_n^{(r)}(x)^2+g_n^{(r)}(y)^2)\,J(x,{\rm d}y)m({\rm d}x)\\
&=\alpha^2 \int_E g_n^{(r)}(x)^2 \left(\int_{0<d(x,y)\le F_r(x,y)}(\rho_r(x)-\rho_r(y))^2\,J(x,{\rm d}y)\right)\,m({\rm d}x)\\
&\le \alpha^2 \sup_{z\in E}\left(\int_{0<d(z,y)\le F_r(z,y)}(\rho_r(z)-\rho_r(y))^2\,J(z,{\rm d}y)\right)\int_E (g_n^{(r)})^2\,{\rm d}m.
\end{split}
\end{equation*}
Therefore, 
\begin{equation}\label{eq:f-bound-M}
{\cal E}^{(c)}(f_n^{(r)},f_n^{(r)})+{\cal E}^{(r),1}(f_n^{(r)},f_n^{(r)})\le \alpha^2 M_1(r)\int_E (g_n^{(r)})^2\,{\rm d}m.
\end{equation}
Since $f_n^{(r)}$ is nonnegative, $F_r(x,y)=F_r(y,x)$ and $J(x,{\rm d}y)m({\rm d}x)=J(y,{\rm d}x)m({\rm d}y)$, 
we obtain 
\begin{equation}\label{eq:big-jump-bound}
\begin{split}
{\cal E}^{(r),2}(f_n^{(r)},f_n^{(r)})
&= \iint_{d(x,y)>F_r(x,y)}( f_n^{(r)}(x) - f_n^{(r)}(y) )^2\,J(x,{\rm d}y)m({\rm d}x)\\
&\le  \iint_{d(x,y)>F_r(x,y)}( f_n^{(r)}(x)^2 + f_n^{(r)}(y)^2 )\,J(x,{\rm d}y)m({\rm d}x)\\
&= 2\int_E f_n^{(r)}(x)^2 \left(\int_{d(x,y)>F_r(x,y)}\,J(x,{\rm d}y)\right)\,m({\rm d}x)\\
&\le 2M_2(r)\int_E (f_n^{(r)})^2\, {\rm d}m.
\end{split}
\end{equation}
By \eqref{eq:decomp} with $u=f_n^{(r)}$, \eqref{eq:f-bound-M} and \eqref{eq:big-jump-bound},  
we arrive at the desired assertion. 
\end{proof}

We finally prove the asymptotic equivalence of the norms of $f_n^{(r)}$ and $g_n^{(r)}$.
\begin{lem}\label{lem:compare}
The sequences $\{f_n^{(r)}\}$ and $\{g_n^{(r)}\}$ satisfy 
$\|f_n^{(r)}\|_{L^2(E;m)}/\|g_n^{(r)}\|_{L^2(E;m)}\rightarrow 1$ as $n\rightarrow \infty$.
\end{lem}

\begin{proof}
By the definitions of $\{f_n^{(r)}\}$ and $\{g_n^{(r)}\}$, 
\begin{equation}\label{eq:decomp-g}
\int_E (g_n^{(r)})^2\,{\rm d}m
=\int_E (f_n^{(r)})^2\,{\rm d}m+4\int_E f_n^{(r)} \,{\rm d}m+4m(K_{\rho_r}(R_n)).
\end{equation}
Then
\[
\int_E (f_n^{(r)})^2\,{\rm d}m
\ge \int_{\rho_r(x)\le R_n/2} (f_n^{(r)})^2\,{\rm d}m=(e^{\alpha R_n/2 }-1)^2m(K_{\rho_r}(R_n/2)).
\]
By the definition of $\mu_r$, we see that for any $\varepsilon>0$, 
there exists $R_*>0$ such that 
\[
m(K_{\rho_r}(R))\ge e^{(\mu_r-\varepsilon )R}, \quad R\ge R_*.
\]
By the definition of $\{R_n\}$, we can also take $N\in {\mathbb N}$ with $R_N\ge R_*$ such that 
\[
m(K_{\rho_r}(R_n))\le e^{( \mu_r + \varepsilon/2 ) R_n}, \quad n\ge N.
\]
Hence if we take $\varepsilon\in ( 0, \alpha-\mu_r/2 )$, then  
\begin{equation}\label{eq:compare-1}
\begin{split}
\frac{m(K_{\rho_r}(R_n))}{\int_E (f_n^{(r)})^2\,{\rm d}m}
&\le \frac{m(K_{\rho_r}(R_n))}{(e^{\alpha R_n/2}-1)^2 m(K_{\rho_r}(R_n/2))}
\le \frac{e^{(\mu_r+\varepsilon/2 ) R_n}}{(e^{\alpha R_n/2}-1)^2e^{(\mu_r-\varepsilon ) R_n/2}}\\
&=\frac{e^{\alpha R_n}}{(e^{\alpha R_n/2}-1)^2}e^{-(\alpha-\mu_r/2-\varepsilon)R_n}\rightarrow 0 \quad (n\rightarrow\infty).
\end{split}
\end{equation}
Since the Cauchy-Schwarz inequality yields 
\[
\int_E f_n^{(r)} \,{\rm d}m
= \int_E f_n^{(r)}{\bf 1}_{K_{\rho_r}(R_n)} \,{\rm d}m
\le \left(\int_E (f_n^{(r)})^2 \,{\rm d}m\right)^{1/2} m(K_{\rho_r}(R_n))^{1/2},
\]
we get by \eqref{eq:compare-1},
\[
\frac{\int_E f_n^{(r)} \,{\rm d}m}{\int_E (f_n^{(r)})^2\,{\rm d}m}
\le \frac{m(K_{\rho_r}(R_n))^{1/2}}{\left(\int_E (f_n^{(r)})^2 \,{\rm d}m\right)^{1/2}}
\rightarrow 0, \quad n\rightarrow\infty.
\]
Combining this with \eqref{eq:decomp-g}, 
we complete the proof. 
\end{proof}

\begin{proof}[Proof of Theorem {\rm \ref{thm:ess}}]
By Lemma \ref{lem:local} (ii), Lemma \ref{lem:upper-bound} (ii) and Lemma \ref{lem:compare},
\[
\lambda_e \le \liminf_{n\rightarrow\infty} \frac{ {\cal E}(f_n^{(r)},f_n^{(r)}) }{ \int_E (f_n^{(r)})^2\,{\rm d}m }
\le \alpha^2 M_1(r) \left( \lim_{n\rightarrow\infty}\frac{ \int_E (g_n^{(r)})^2\,{\rm d}m }{ \int_E (f_n^{(r)})^2\,{\rm d}m } \right) + 2 M_2(r) 
=\alpha^2 M_1(r)+2M_2(r).
\]
Since $r>0$ and $\alpha>\mu_r /2$ are arbitrary, we arrive at the desired assertion.
\end{proof}

\subsection{Finite volume}\label{subsect:finite}
In this subsection, we assume that $m(E)<\infty$. 
For $r>0$, let 
\[
\nu_r=\liminf_{R \rightarrow \infty}\frac{-1}{R}\log m(K_{\rho_r}(R)^c).
\]

\begin{thm}\label{thm:ess-f}
Assume that $\nu_r<\infty$ for some $r>0$. 
If $({\cal E},{\cal F})$ is recurrent, then 
\[
\lambda_e \le \inf_{r>0} \left( \frac{\nu_r^2}{4} M_1(r)+2M_2(r) \right).
\]
\end{thm}

We note that if $m(E)<\infty$ and $({\cal E},{\cal F})$ is recurrent, 
then any constant function belongs to ${\cal F}$ by \cite[Theorem 1.6.3 and Theorem 1.5.2 (iii)]{FOT11}.

Let us prove Theorem \ref{thm:ess-f} 
by following the arguments of \cite{B84, HKW13} and Theorem \ref{thm:ess}. 
In what follows, we assume that  $\nu_r<\infty$ for some $r>0$. 
Then there exists a sequence $\{R_n\}$ such that $R_n \rightarrow \infty$ as $n \rightarrow \infty$ and
\[
\nu_r=\lim_{n\rightarrow\infty}\frac{-1}{R_n}\log m(K_{\rho_r}(R_n)^c).
\]
We fix such a sequence $\{R_n\}$.

For any $\alpha>\nu_r/2$, let $w_n \ (n\ge 1)$ be a continuous function on $[0,\infty)$ such that 
\[
w_n(t)=
\begin{dcases}
0, & 0 \le t \le R_n/2, \\
\alpha(t-R_n/2), & R_n/2 < t \le R_n, \\
\alpha R_n/2, & t > R_n. 
\end{dcases}
\]
Define $h_n^{(r)}(x)=w_n(\rho_r(x))$ and $f_n^{(r)}(x)=e^{h_n^{(r)}(x)}-1$. 
We first prove that 
we can take the sequence $\{f_n^{(r)}/\|f_n^{(r)}\|_{L^2(E;m)}\}$ 
as $\{f_n\}$ in  Proposition \ref{prop:persson}.

\begin{lem}\label{lem:local-f}
Assume that $({\cal E},{\cal F})$ is recurrent. 
\begin{enumerate}
\item[{\rm (i)}] For any $n\ge 1$, $h_n^{(r)}$ and $f_n^{(r)}$ belong to ${\cal F}$.
\item[{\rm (ii)}] 
The sequence $\{ f_n^{(r)}/\|f_n^{(r)}\|_{L^2(E;m)} \}$ is weakly convergent to $0$ in $L^2(E;m)$. 
In particular, 
\begin{equation}\label{eq:persson-2}
\lambda_e \le \liminf_{n\rightarrow\infty}\frac{{\cal E}(f_n^{(r)},f_n^{(r)})}{\|f_n^{(r)}\|_{L^2(E;m)}^2}.
\end{equation}
\end{enumerate}
\end{lem}

\begin{proof}
We first prove (i). 
Since $\rho_r\in {\cal F}_{{\rm loc}}$, 
we can follow the proof of \cite[Lemma 3.1]{S15} to show that $\alpha R_n/2 - h_n^{(r)}\in {\cal F}\cap C_0(E)$. 
As any constant function belongs to ${\cal F}$ by assumption,  
we have $h_n^{(r)}\in {\cal F}$.
Following the proof of Lemma \ref{lem:local}, 
we further obtain $f_n^{(r)}\in {\cal F} \cap C_0(E)$.

We next prove (ii). 
Take any $u\in L^2(E;m)$ with $u\ne 0$. 
Then for any $R>0$,  we have by the Cauchy-Schwarz inequality, 
\begin{equation}\label{eq:schwarz-f}
\begin{split}
\left|\int_E u f_n^{(r)} \,{\rm d}m\right|
&=\left|\int_E u f_n^{(r)} {\bf 1}_{K_{\rho_r}(R)} \, {\rm d}m + \int_E u f_n^{(r)} {\bf 1}_{K_{\rho_r}(R)^c} \, {\rm d}m\right|\\
&\le \left|\int_E u f_n^{(r)} {\bf 1}_{K_{\rho_r}(R)} \, {\rm d}m \right| + \left| \int_E u f_n^{(r)} {\bf 1}_{K_{\rho_r}(R)^c} \, {\rm d}m \right|\\
&\le \|u\|_{L^2(E;m)} \left(\int_E (f_n^{(r)})^2{\bf 1}_{K_{\rho_r}(R)}\,{\rm d}m\right)^{1/2}
+ \|f_n^{(r)}\|_{L^2(E;m)}\left(\int_E u^2{\bf 1}_{K_{\rho_r}(R)^c}\,{\rm d}m\right)^{1/2}. 
\end{split}
\end{equation}
On the other hand, for any $\varepsilon>0$, there exists $R_*>0$ such that 
\[
\left(\int_E u^2{\bf 1}_{K_{\rho_r}(R_*)^c}\,{\rm d}m \right)^{1/2} < \varepsilon. 
\]
In particular, for any $n\in {\mathbb N}$ with $R_n\ge 2 R_*$, 
we have $f_n^{(r)}{\bf 1}_{K_{\rho_r}(R_*)}=0$ and so 
\[
\int_E (f_n^{(r)})^2{\bf 1}_{K_{\rho_r}(R_*)}\,{\rm d}m=0.
\]
Hence if we take $R=R_*$ in \eqref{eq:schwarz-f}, then 
\[
\left| \int_E u f_n^{(r)} \,{\rm d}m \right|
\le \|f_n^{(r)}\|_{L^2(E;m)} \left( \int_E u^2{\bf 1}_{K_{\rho_r}(R_*)^c}\,{\rm d}m \right)^{1/2}
 <\varepsilon \|f_n^{(r)}\|_{L^2(E;m)}.
\]
Namely, $\{ f_n^{(r)}/\|f_n^{(r)}\|_{L^2(E;m)} \}$ is weakly convergent to $0$ in $L^2(E;m)$. 
Combining this with Proposition \ref{prop:persson}, we further get \eqref{eq:persson-2}.
\end{proof}

We next discuss the upper bound of ${\cal E}(f_n^{(r)},f_n^{(r)})$. 
Let $g_n^{(r)}=(f_n^{(r)}+2){\bf 1}_{K_{\rho_r}(R_n/2)^c}$, 
and let $\varphi$ be as in \eqref{eq:varphi}. 
Following the proof of \cite[Lemma 2.5]{HKW13}, we obtain  
\begin{lem}\label{lem:f-bound-f}
\begin{enumerate}
\item[{\rm (i)}] 
For any $n\ge 1$, 
\begin{equation}\label{eq:f-bound-f}
(f_n^{(r)}(x)-f_n^{(r)}(y))^2
\le \varphi(|\rho_r(x)-\rho_r(y)|) (g_n^{(r)}(x)^2+g_n^{(r)}(y)^2), \quad x,y\in E.
\end{equation}
In particular, 
\[
(f_n^{(r)}(x)-f_n^{(r)}(y))^2
\le \frac{\alpha^2}{2}(\rho_r(x)-\rho_r(y))^2 (g_n^{(r)}(x)^2+g_n^{(r)}(y)^2), \quad x,y\in E.
\]
\item[{\rm (ii)}] For any $n\ge 1$, 
\[
{\cal E}(f_n^{(r)},f_n^{(r)})
\le \alpha^2 M_1(r) \int_E (g_n^{(r)})^2\,{\rm d}m+ 2M_2(r) \int_E (f_n^{(r)})^2\,{\rm d}m.
\]
\end{enumerate}
\end{lem}

\begin{proof}
We first prove (i). By symmetry, we may and do assume that $\rho_r(x)\le \rho_r(y)$. 
\begin{enumerate}

\item[(a)] Assume that $\rho_r(x)\le \rho_r(y)\le R_n/2$ or $R_n\le \rho_r(x)\le \rho_r(y)$. 
Then by definition, we have $f_n^{(r)}(x)=f_n^{(r)}(y)$ and so \eqref{eq:f-bound-f} follows. 

\item[(b)] Assume that $\rho_r(x)\le R_n/2 \le \rho_r(y) \le R_n$. 
Then 
\begin{equation*}
\begin{split}
\left(f_n^{(r)}(x)-f_n^{(r)}(y)\right)^2
&=(e^{\alpha (\rho_r(y) - R_n/2)}-1)^2 \\
&=\frac{(e^{\alpha (\rho_r(y) - R_n/2)}-1)^2}{1 + e^{ 2 \alpha ( \rho_r(y) - R_n/2)}}
 (1 + e^{ 2 \alpha (\rho_r(y) - R_n/2)}).
\end{split}
\end{equation*}
Since 
\[
\frac{(e^{\alpha (\rho_r(y) - R_n/2)}-1)^2}{1 + e^{ 2 \alpha ( \rho_r(y) - R_n/2)}}
=\varphi(\rho_r(y)-R_n/2) \le \varphi(\rho_r(y)-\rho_r(x))
\]
and 
\[
1 + e^{ 2 \alpha (\rho_r(y) - R_n/2)} \le  g_n^{(r)}(x)^2+g_n^{(r)}(y)^2,
\]
we obtain \eqref{eq:f-bound-f}.

\item[(c)] 
Assume that $R_n/2 \le \rho_r(x)\le \rho_r(y)\le R_n$. Then 
\begin{equation*}
\begin{split}
\left(f_n^{(r)}(x)-f_n^{(r)}(y)\right)^2
&=(e^{ \alpha (\rho_r(x) - R_n/2)} - e^{\alpha (\rho_r(y) - R_n/2)})^2\\
&=\frac{(e^{ \alpha (\rho_r(x) - R_n/2)} - e^{\alpha (\rho_r(y) - R_n/2)})^2}
{e^{2 \alpha ( \rho_r(x) - R_n/2 )} + e^{2 \alpha ( \rho_r(y) - R_n/2 )}}
(e^{2 \alpha ( \rho_r(x) p- R_n/2 )} + e^{2 \alpha ( \rho_r(y) - R_n/2 )}) \\
&=\frac{(1 - e^{\alpha (\rho_r(y) - \rho_r(x))})^2}{1+e^{2 \alpha ( \rho_r(y) - \rho_r(x) )}} 
(e^{2 \alpha ( \rho_r(x) - R_n/2 )} + e^{2 \alpha ( \rho_r(y) - R_n/2)})\\
&\le \varphi(\rho_r(y)-\rho_r(x)) (g_n^{(r)}(x)^2+g_n^{(r)}(y)^2).
\end{split}
\end{equation*}

\item[(d)]
Assume  that $\rho_r(x)\le R_n/2 \le R_n \le \rho_r(y)$. 
Then 
\[
\left(f_n^{(r)}(x)-f_n^{(r)}(y)\right)^2
=(e^{\alpha R_n/2}-1)^2
=\frac{(e^{\alpha R_n/2}-1)^2}{e^{\alpha R_n }+1}(e^{\alpha R_n}+1)=\varphi(R_n/2)(e^{\alpha R_n}+1).
\]
Since $R_n/2 \le \rho_r(y)-\rho_r(x)$ and $e^{\alpha R_n}+1\le g_n^{(r)}(x)^2+g_n^{(r)}(y)^2$, 
we have \eqref{eq:f-bound-f}.

\item[(e)] Assume that $R_n/2 \le \rho_r(x) \le R_n \le \rho_r(y)$. 
Then 
\begin{equation*}
\begin{split}
\left(f_n^{(r)}(x)-f_n^{(r)}(y)\right)^2
&=(e^{\alpha ( \rho_r(x) - R_n/2)}-e^{ \alpha R_n/2 })^2 \\
&=\frac{(e^{\alpha ( \rho_r(x) - R_n/2)}-e^{ \alpha R_n/2 })^2}{ e^{\alpha R_n}+e^{2 \alpha ( \rho_r(x) - R_n/2 )}}
(e^{\alpha R_n}+e^{2 \alpha ( \rho_r(x) - R_n/2 )})\\
&=\frac{(e^{\alpha ( \rho_r(x) - R_n)}-1)^2}{ 1+e^{2 \alpha ( \rho_r(x) - R_n )} }
(e^{\alpha R_n}+e^{2 \alpha ( \rho_r(x) - R_n/2 )}).
\end{split}
\end{equation*}
Since 
\[
\frac{(e^{\alpha ( \rho_r(x) - R_n)}- 1)^2}{ 1+e^{2 \alpha ( \rho_r(x) - R_n )} }
= \varphi( \rho_r(x) - R_n) = \varphi( R_n - \rho_r(x)) \le \varphi ( \rho_r(y) - \rho_r(x)) 
\]
and 
\[
e^{\alpha R_n}+e^{2 \alpha ( \rho_r(x) - R_n/2 )} 
\le g_n^{(r)}(x)^2+g_n^{(r)}(y)^2, 
\]
we get \eqref{eq:f-bound-f}.
\end{enumerate}
By the argument above, the proof of (i) is complete.

We omit the proof of (ii) because it is the same with that of Lemma \ref{lem:upper-bound} (ii). 
\end{proof}

We finally prove the asymptotic equivalence of the norms of $f_n^{(r)}$ and $g_n^{(r)}$.
\begin{lem}\label{lem:compare-f}
The sequences $\{f_n^{(r)}\}$ and $\{g_n^{(r)}\}$ satisfy 
$\|f_n^{(r)}\|_{L^2(E;m)}/\|g_n^{(r)}\|_{L^2(E;m)}\rightarrow 1$ as $n\rightarrow \infty$.
\end{lem}

\begin{proof}
By the definitions of $\{f_n^{(r)}\}$ and $\{g_n^{(r)}\}$, 
\begin{equation}\label{eq:decomp-g-f}
\int_E (g_n^{(r)})^2\,{\rm d}m
=\int_E (f_n^{(r)})^2\,{\rm d}m+4\int_E f_n^{(r)} \,{\rm d}m+4m(K_{\rho_r}(R_n/2)^c).
\end{equation}
Then
\[
\int_E (f_n^{(r)})^2\,{\rm d}m
\ge \int_{\rho_r(x)>R_n} (f_n^{(r)})^2\,{\rm d}m=(e^{\alpha R_n /2}-1)^2m(K_{\rho_r}(R_n)^c).
\]
By the definition of $\nu_r$, we see that for any $\varepsilon>0$, 
there exists $R_*>0$ such that 
\[
m(K_{\rho_r}(R)^c) \le e^{-(\nu_r-\varepsilon )R}, \quad R\ge R_*.
\]
By the definition of $\{R_n\}$, we can also take $N\in {\mathbb N}$ with $R_N\ge R_*$ such that 
\[
m(K_{\rho_r}(R_n)^c) \ge e^{-( \nu_r + \varepsilon/2 ) R_n}, \quad n\ge N.
\]
Hence if we take $\varepsilon\in ( 0, \alpha-\nu_r/2 )$, then  
\begin{equation}\label{eq:compare-1-f}
\begin{split}
\frac{m(K_{\rho_r}(R_n/2)^c)}{\int_E (f_n^{(r)})^2\,{\rm d}m}
&\le \frac{m(K_{\rho_r}(R_n/2)^c)}{(e^{\alpha R_n/2}-1)^2 m(K_{\rho_r}(R_n)^c)}
\le \frac{e^{- (\nu_r - \varepsilon ) R_n/2}}{(e^{\alpha R_n/2}-1)^2 e^{- (\nu_r +  \varepsilon /2) R_n}}\\
&=\frac{e^{\alpha R_n}}{(e^{\alpha R_n/2}-1)^2} e^{-(\alpha-\nu_r/2-\varepsilon )R_n}\rightarrow 0, \quad n\rightarrow\infty.
\end{split}
\end{equation}
Since the Cauchy-Schwarz inequality yields 
\[
\int_E f_n^{(r)} \,{\rm d}m
= \int_E f_n^{(r)}{\bf 1}_{K_{\rho_r}(R_n/2)^c} \,{\rm d}m
\le \left(\int_E (f_n^{(r)})^2 \,{\rm d}m\right)^{1/2} m(K_{\rho_r}(R_n/2)^c)^{1/2},
\]
we get by \eqref{eq:compare-1-f},
\[
\frac{\int_E f_n^{(r)} \,{\rm d}m}{\int_E (f_n^{(r)})^2\,{\rm d}m}
\le \frac{m(K_{\rho_r}(R_n/2)^c)^{1/2}}{\left(\int_E (f_n^{(r)})^2 \,{\rm d}m\right)^{1/2}}
\rightarrow 0, \quad n\rightarrow\infty.
\]
Combining this with \eqref{eq:decomp-g-f} and \eqref{eq:compare-1-f}, 
we complete the proof. 
\end{proof}

\begin{proof}[Proof of Theorem {\rm \ref{thm:ess-f}}]
We can follow the proof of Theorem \ref{thm:ess}  
by using Lemma \ref{lem:local-f} (ii), Lemma \ref{lem:f-bound-f} (ii) and Lemma \ref{lem:compare-f}. 
\end{proof}

\section{Volume growth}\label{sect:volume}
In this section, we are concerned with the relation between 
the volume growth and the upper bound of the bottom of the essential spectrum. 
Let $K_x(r)=\{ y\in E \mid d(x,y)\le r\}$ be a closed ball with center $x\in E$ and radius $r\ge 0$.
Throughout this section, we impose the following assumption on the regular Dirichlet form $({\cal E},{\cal F})$. 
\begin{assum}\label{assum:length-1} \rm 
$({\cal E},{\cal F})$ is a regular Dirichlet form on $L^2(E;m)$ satisfying Assumption \ref{assum:BD} 
and the next conditions:
\begin{enumerate}
\item[(i)] There exists a positive symmetric measurable function $J(x,y)$ on $E\times E$ 
such that $J(x,{\rm d}y)=J(x,y)\,m({\rm d}y)$.
\item[(ii)] For some $o\in E$, the function $d_0(x)=d(o,x) \ (x\in E)$ belongs to ${\cal F}_{{\rm loc}}$.
\item[(iii)] For any $x\in E$ and $r>0$, the closed ball $K_x(r)$ is compact in $E$. 
\end{enumerate}
\end{assum}

\subsection{Polynomial volume growth}
In this subsection, 
we discuss the upper bound of $\lambda_e$ under the next conditions:
\begin{itemize}
\item 
The measure $m$ satisfies $m(E)=\infty$, and for some positive constants $C_1$ and $\eta$,  
\[
m(K_x(r)) \le C_1 r^{\eta}, \quad x\in E, \ r>0.
\]
\item 
There exist positive constants $C_2$, $C_3$, $\eta$, $\beta_1 \ (0<\beta_1<2)$ and $\beta_2$ such that 
for any $x,y\in E$, 
\[
J(x,y) \le 
\begin{dcases}
\frac{C_2}{d(x,y)^{\eta+\beta_1}}, & d(x,y)\le 1, \\
\frac{C_3}{d(x,y)^{\eta+\beta_2}}, & d(x,y)> 1.
\end{dcases}
\]
\end{itemize}

Take $\rho_r(x)=d_0(x)$ and $F_r(x,y)=r$, and so  $\mu_r=0$. 
We first calculate the upper bound of $M_1(r)$. 
By definition, 
\begin{equation}\label{eq:m_1-upper}
\begin{split}
M_1(r)
&=\esssup_{x\in E}\int_{0<d(x,y)\le r} (d_0(x)-d_0(y))^2 \,J(x,y)m({\rm d}y)\\
&\le \sup_{x\in E}\int_{0<d(x,y)\le r}d(x,y)^2 \,J(x,y)m({\rm d}y).
\end{split}
\end{equation}
For $r\in (0,1]$,
\[
\int_{0<d(x,y)\le r}d(x,y)^2 \,J(x,y)m({\rm d}y)
\le C_2 \int_{0<d(x,y)\le r}d(x,y)^{2-(\eta+\beta_1)}\,m({\rm d}y).
\]
Let $V_x(r)=m(K_x(r))$. Since 
\begin{equation}\label{eq:ibp-1}
\begin{split}
&\int_{0<d(x,y)\le r}d(x,y)^{2-(\eta+\beta_1)}\,m({\rm d}y)
=\int_{(0,r]} s^{2-(\eta+\beta_1)}\, {\rm d}V_x(s)\\
&=[s^{2-(\eta+\beta_1)}V_x(s)]_{s=0}^{s=r} - (2-(\eta+\beta_1))\int_0^r s^{1-(\eta+\beta_1)}V_x(s)\,{\rm d}s
\le c_1r^{2-\beta_1},
\end{split}
\end{equation}
we obtain $M_1(r)\le c_1r^{2-\beta_1}$ for $r\in (0,1]$. 
In the same way, we have $M_1(r)\le c_2 r^{2-\beta_2}$ for $r>1$. 

We next calculate the upper bound of $M_2(r)$. 
For $r>1$, we have as in \eqref{eq:ibp-1},
\begin{equation}\label{eq:ibp-2}
\int_{d(x,y) > r} J(x,y)m({\rm d}y)
\le C_2\int_{d(x,y) > r}d(x,y)^{-(\eta+\beta_2)}m({\rm d}y)
\le c_3r^{-\beta_2},
\end{equation}
which implies that $M_2(r)\le c_3 r^{-\beta_2}$. 
For $r\in (0,1]$, 
\begin{equation*}
\begin{split}
\int_{d(x,y) > r} J(x,y)m({\rm d}y)
&\le  C_1\int_{r<d(x,y) \le 1}d(x,y)^{-(\eta+\beta_1)}m({\rm d}y) + C_2\int_{d(x,y) > 1}d(x,y)^{-(\eta+\beta_2)}m({\rm d}y) \\
&\le c_4 r^{-\beta_1},
\end{split}
\end{equation*}
which implies that $M_2(r)\le c_4 r^{-\beta_1}$. 

By the argument above, we obtain
\[
\inf_{r>0}\left( \frac{\mu_r^2}{4} M_1(r) + M_2(r) \right)=0
\]
and so $\lambda_e=0$ by Theorem \ref{thm:ess}.

\subsection{Exponential volume growth}

In this subsection, 
we discuss the upper bound of $\lambda_e$ under the next conditions:
\begin{itemize}
\item The measure $m$ satisfies $m(E)=\infty$, 
and there exist positive constants $C_1$, $C_2$, $\eta$ and $\kappa$ such that for any $x\in E$,  
\[
m(K_x(r))\le 
\begin{dcases}
C_1r^{\eta}, & 0< r \le 1 , \\
C_2e^{\kappa r}, & r>1.
\end{dcases}
\]
\item There exist positive constants $C_3$, $C_4$, $\beta_1 \ (0<\beta_1<2)$, $\beta_2$ 
and $\lambda \ge \kappa$ such that 
\[
J(x,y)\le 
\begin{dcases}
\frac{C_3}{d(x,y)^{\eta+\beta_1}}, & 0<d(x,y) \le 1, \\
\frac{C_4 e^{-\lambda d(x,y)}}{d(x,y)^{\beta_2}}, & d(x,y) > 1.
\end{dcases}
\]
\end{itemize}
This formulation is the same as  \cite[Example 5.7]{GHM12},  
which is motivated by the fractional Laplacian on the hyperbolic space. 
We will explain details about this matter in Remark \ref{rem:hyperbolic} below. 

Take $\rho_r(x)=d_0(x)$ and $F_r(x,y)=r$. 
We first calculate the upper bound of $M_1(r)$. 
As in \eqref{eq:m_1-upper},  
\[
M_1(r)
\le \sup_{x\in E}\int_{0<d(x,y)\le r}d(x,y)^2 \,J(x,y)m({\rm d}y).
\]
Then by following the calculation just after \eqref{eq:m_1-upper}, 
we have $M_1(r)\le c_1 r^{2-\beta_1}$ for $r\in (0,1]$. 
Assume that $r>1$. Let
\begin{equation*}
\begin{split}
&\int_{0<d(x,y)\le r} d(x,y)^2 \,J(x,y)m({\rm d}y)\\
&= \int_{0<d(x,y)\le 1} d(x,y)^2 \,J(x,y)m({\rm d}y) 
+\int_{1<d(x,y)\le r} d(x,y)^2 \,J(x,y)m({\rm d}y)\\
&={\rm (I)}+{\rm (II)}.
\end{split}
\end{equation*}
Then by the same argument as for $r\in (0,1]$, we have ${\rm (I)}\le c_1$. 
Let $V_x(r)=m(K_x(r))$. Then 
\[
{\rm (II)} 
\le C_4  \int_{1<d(x,y)\le r} e^{-\lambda d(x,y)}d(x,y)^{2-\beta_2}\, m({\rm d}y)
=C_4\int_{(1,r]} e^{-\lambda s} s^{2-\beta_2} \, {\rm d}V_x(s).
\]
By the integration by parts formula, we obtain 
\begin{equation}\label{eq:hyp-1}
\begin{split}
\int_{(1,r]} e^{-\lambda s} s^{2-\beta_2} \, {\rm d}V_x(s)
&\le c_2 \int_1^r e^{-\lambda s} s^{2-\beta_2}V_x(s)\,{\rm d}s \\
&\lesssim 
\begin{dcases}
r^{3-\beta_2} & \text{if $\lambda=\kappa$ and  $0<\beta_2<3$}, \\
\log r & \text{if $\lambda=\kappa$ and $\beta_2=3$}, \\
1 & \text{if $\lambda>\kappa$ or $\beta_2>3$}.
\end{dcases}
\end{split}
\end{equation}
Hence for $r>1$, 
\[
M_1(r) \lesssim 
\begin{dcases}
r^{3-\beta_2} & \text{if $\lambda=\kappa$ and  $0<\beta_2<3$}, \\
\log r & \text{if $\lambda=\kappa$ and $\beta_2=3$}, \\
1 & \text{if $\lambda>\kappa$ or $\beta_2>3$}.
\end{dcases}
\]

We turn to the upper bound of $M_2(r)$. 
If $r>1$, 
then by the similar calculation as for \eqref{eq:hyp-1}, we get 
\begin{equation*}
\begin{split}
\int_{d(x,y)>r} J(x,y)\,m({\rm d}y)
&\le C_4 \int_{d(x,y)>r} \frac{e^{-\lambda d(x,y)}}{d(x,y)^{\beta_2}}\,m({\rm d}y)
= C_4 \int_{(r, \infty)} e^{-\lambda s} s^{-\beta_2}\,{\rm d}V_x(s)\\
&\lesssim  
\begin{dcases}
e^{-(\lambda-\kappa) r} r^{-\beta_2} & \text{if $\lambda>\kappa$},\\
r^{-(\beta_2-1)} & \text{if $\lambda=\kappa$ and  $\beta_2>1$}.
\end{dcases}
\end{split}
\end{equation*}
Assume that $0<r\le 1$. 
Then 
\[
\int_{d(x,y)>r} J(x,y)\,m({\rm d}y)
=\int_{r<d(x,y) \le 1} J(x,y)\,m({\rm d}y)+\int_{d(x,y)>1} J(x,y)\,m({\rm d}y).
\]
By the calculation as for $r>1$, the second term above is finite 
if $\lambda>\kappa$, or if $\lambda=\kappa$ and $\beta_2>1$. 
We also have 
\[
\int_{r<d(x,y) \le 1} J(x,y)\,m({\rm d}y) 
\le C_3 \int_{r<d(x,y)\le 1}\frac{1}{d(x,y)^{\eta+\beta_1}}\,m({\rm d}y)
=C_3 \int_r^1 \frac{1}{s^{\eta + \beta_1}}\,{\rm d}V_x(s).
\]
Then by the integration by parts formula, 
\[
\int_r^1 \frac{1}{s^{\eta + \beta_1}}\,{\rm d}V_x(s)
\le V_x(1) + C_1 (\eta + \beta_1) \int_r^1 \frac{s^{\eta}}{s^{\eta + \beta_1 + 1}}\,{\rm d}s 
 \asymp r^{-\beta_1}.
\]
Hence the argument above implies that if $\lambda>\kappa$, or if $\lambda=\kappa$ and $\beta_2>1$, 
then 
\begin{itemize}
\item For any $r \in (0,1]$,  
\[
M_2(r) \lesssim  r^{-\beta_1}.
\]
\item For any $r>1$, 
\[
M_2(r)
\lesssim 
\begin{dcases}
e^{-(\lambda-\kappa)r} r^{-\beta_2} & \text{if $\lambda>\kappa$},\\
 r^{-(\beta_2-1)} & \text{if $\lambda=\kappa$ and  $\beta_2>1$}.
\end{dcases}
\] 
\end{itemize}

We now focus on the condition that 
$\lambda=\kappa$, $\beta_1=\alpha$ and $\beta_2=1+\alpha/2$ for some $\alpha\in (0,2)$. 
Then for some $C>0$, 
\[
M_1(r)\le C\times 
\begin{dcases}
r^{2-\alpha}, & 0<r\le 1, \\
r^{2-\alpha/2}, & r>1
\end{dcases}
\]
and 
\[
M_2(r)\le C\times 
\begin{dcases}
r^{-\alpha}, & 0<r\le 1, \\
r^{-\alpha/2}, & r>1.
\end{dcases}
\]
Therefore, 
\begin{equation}\label{eq:stable-upper}
\begin{split}
\lambda_e 
&\le \inf_{r>0}\left(\frac{\kappa^2}{4}M_1(r)+2 M_2(r)\right) \\
&\le C\times 
\begin{dcases}
\left(\frac{2}{2-\alpha}\right)^{1-\alpha/4}\left(\frac{\kappa}{2}\right)^{\alpha} 
& \text{if $0<\kappa \le 2\sqrt{\alpha/(2-\alpha)}$}, \\
\left(\frac{2}{2-\alpha}\right)^{1-\alpha/2}\left(\frac{\kappa}{2}\right)^{\alpha} 
& \text{if $\kappa > 2\sqrt{\alpha/(2-\alpha)}$}.
\end{dcases}
\end{split}
\end{equation}

\begin{rem}\label{rem:hyperbolic}\rm 
For $n\ge 2$, let ${\mathbb H}^n$ be the $n$-dimensional hyperbolic space, 
and let $\Delta$ be the Laplace-Beltrami operator on ${\mathbb H}^n$. 
We define the distance ball $K(r)=\{x\in {\mathbb H}^n \mid \rho(o,x)\le r\}$ for some $o\in {\mathbb H}^n$ and $r>0$.
Let $m$ be the Riemannian volume measure on ${\mathbb H}^n$. 
It is known that 
\begin{equation}\label{eq:hyp-vol}
m(K(R))=\omega_n \int_0^R (\sinh t)^{n-1}\,{\rm d}t\sim c_n e^{(n-1)R}, \quad R\rightarrow\infty
\end{equation}
and 
\begin{equation}\label{eq:hyp-ess}
\inf \sigma_{{\rm ess}}\left(-\frac{1}{2}\Delta \right)=\frac{(n-1)^2}{8}
\end{equation}
(see, e.g., \cite[Section 5.7]{D92}).
Note that the constant $n-1$ in \eqref{eq:hyp-ess} 
coincides with  the exponential volume growth rate of ${\mathbb H}^n$ in \eqref{eq:hyp-vol}. 

For $\alpha\in (0,2)$, let $({\cal E}^{\alpha},{\cal F}^{\alpha})$ be
a regular Dirichlet form on $L^2({\mathbb H}^n;m)$ subordinate to $({\cal E},{\cal F})$ 
with respect to the $\alpha/2$-subordinator. 
Then $({\cal E}^{\alpha},{\cal F}^{\alpha})$ is non-local, 
and its generator is formally written as $-(-\Delta/2)^{\alpha/2}$. 
\begin{enumerate}
\item[(i)] 
By \eqref{eq:hyp-ess}, we have for any $\alpha\in (0,2)$,  
\begin{equation}\label{eq:hyp-frac-ess}
\inf \sigma_{{\rm ess}}\left(\left(-\frac{1}{2}\Delta \right)^{\alpha/2}\right)
=\left\{\frac{(n-1)^2}{8}\right\}^{\alpha/2}
=\frac{(n-1)^{\alpha}}{2^{3\alpha/2}}.
\end{equation}
On the other hand, we see by \cite[Lemma 3.1]{RZ16} that  
\begin{equation}\label{eq:hyp-2}
J(x,y)\asymp 
\begin{dcases}
\frac{1}{d(x,y)^{n+\alpha}}, & d(x,y)<1, \\
\frac{e^{-(n-1)d(x,y)}}{d(x,y)^{\alpha}(1+d(x,y)^{1-\alpha/2})}, & d(x,y)\ge 1.
\end{dcases}
\end{equation}
Therefore, Theorem \ref{thm:ess} is applicable to the Dirichlet form $({\cal E}^{\alpha},{\cal F}^{\alpha})$ 
with the following parameters:
\[
\kappa=\lambda=n-1, \ \gamma=n, \ \beta_1=\alpha, \ \beta_2=1+\frac{\alpha}{2}.
\]
In particular, we have \eqref{eq:stable-upper} with $\kappa=n-1$, 
which might seem compatible with \eqref{eq:hyp-frac-ess}.
However, we do not know how the constant $C$ in \eqref{eq:stable-upper} depends on the parameter $\kappa$. 
Since the constant $C$ may affect the upper bound of $\lambda_e$ in \eqref{eq:stable-upper}, 
it is unclear whether \eqref{eq:stable-upper} is sharp or not in terms of the exponential volume growth rate.

\item[(ii)] 
Let $C_0^{{\rm lip}}({\mathbb H}^n)$ be 
the totality of Lipschitz continuous functions on ${\mathbb H}^n$ with compact support. 
Then by \cite[Theorem 2.1]{O02}, 
${\cal F}^{\alpha}\subset {\cal F}^{\beta}$ holds for any $\alpha, \beta\in (0,2]$ with $\alpha\le \beta$, 
and $C_0^{{\rm lip}}({\mathbb H}^n)$ is a core of ${\cal F}^{\alpha}$ for any $\alpha\in (0,2]$. 
Hence by \cite[Theorem 7.3]{FLW14} and the calculation similar to \cite[Subsection 14.4]{FLW14}, 
we see that for any $p\in (0,\alpha/4)$, 
there exists $c>0$ such that 
the metric $\rho(x,y):=c(d(x,y)\wedge d(x,y)^{p})$ is an intrinsic metric for $({\cal E}^{\alpha}, {\cal F}^{\alpha})$ 
in the sense of \cite{FLW14}.
However, if we define the $\rho$-distance ball $K_{\rho}(r)=\{x\in {\mathbb H}^d \mid \rho(o,x)\le r\}$ for $o\in {\mathbb H}^d$ and $r>0$, 
then \eqref{eq:hyp-vol} implies that for all large $R\ge 1$,
\[
m(K_{\rho}(R))\asymp \exp\left(c^{-1/p}(n-1)R^{1/p}\right).
\]
Since $0<p<1$, we have 
\begin{equation}\label{eq:super-exp}
\lim_{R\rightarrow\infty}\frac{1}{R}\log m(K_{\rho}(R))=\infty.
\end{equation}

By taking into consideration the calculation in \cite[Subsection 14.4]{FLW14}, 
we can regard the distance $\rho$ above as a natural  intrinsic metric for $({\cal E}^{\alpha},{\cal F}^{\alpha})$. 
However, since \eqref{eq:super-exp} holds, 
it would not follow from \cite[Theorem 1.1]{HKW13} that for $({\cal E}^{\alpha},{\cal F}^{\alpha})$, 
the essential spectrum is non-empty. 
\end{enumerate}
\end{rem}

\section{Coefficient growth}\label{sect:coeff}
In this section, we focus on the relation between 
the coefficient growth and the bottom of the essential spectrum. 
We here include the coefficient in the jump kernel, 
or in the underlying measure of non-local Dirichlet forms.
The latter formulation is nothing but the time change of Dirichlet forms. 
Throughout this section, we keep Assumption \ref{assum:length-1}.

\subsection{Coefficient in the jump kernel}\label{subsect:coeff}
In this subsection, we examine how the coefficient in the jump kernel affects 
the upper bound of the bottom of the essential spectrum. 
Let $p\in [0,2]$ and $q\in [0,2)$, and let 
\begin{equation*}
\begin{split}
c(x,y)
&=\{(1+d_0(x))^p+(1+d_0(y))^p\}{\bf 1}_{d(x,y) \le 1}\\
&+\{(1+d_0(x))^q+(1+d_0(y))^q\}{\bf 1}_{d(x,y) > 1}, \quad   x,y\in E.
\end{split}
\end{equation*}
We impose the next conditions on  the volume growth and jump kernel.
\begin{itemize}
\item 
There exist positive constants $C_1$ and $\eta$ such that 
\[
m(K_x(r)) \le C_1 r^{\eta}, \quad x\in E, \ r>0.
\]
\item
There exist positive constants $C_2$ and $\beta\in (q,2)$ such that 
\[
J(x,y) \le C_2 \frac{c(x,y)}{d(x,y)^{\eta+\beta}}, \quad (x,y)\in E\times E \setminus {\rm diag}.
\]
\end{itemize}

\begin{enumerate}
\item[(i)] We first prove that if $p<2$, then $\lambda_e=0$. 
Assume that $p<2$. 
Let $\delta>0$ satisfy $p<2(1-\delta)$ and $q<\beta(1-\delta)$.
For fixed constants $r>0$ and $c_*\in (0,1)$, we define 
\[
\rho_r(x)=(1+r+d_0(x))^{\delta}, \quad x\in E
\]
and 
\[
F_r(x,y)=c_* \{r+(d_0(x)\vee d_0(y))\}^{1-\delta}, \quad x,y\in E.
\]
Since 
\begin{equation*}
\begin{split}
|(1+r+t)^{\delta}-(1+r+s)^{\delta}|
&=\delta \int_{s\wedge t}^{s\vee t} \frac{1}{(1+r+u)^{1-\delta}}\,{\rm d}u\\
&\le \frac{\delta |t-s|}{(1+r+(s\wedge t))^{1-\delta}}, \quad s,t>0,
\end{split}
\end{equation*}
we have
\begin{equation}\label{eq:length-1}
\begin{split}
|\rho_r(x)-\rho_r(y)|
&\le \frac{\delta |d_0(x)-d_0(y)|}{(1+r+(d_0(x)\wedge d_0(y)))^{1-\delta}}\\
&\le \frac{\delta d(x,y)}{(1+r+(d_0(x)\wedge d_0(y)))^{1-\delta}}, \quad x,y\in E.
\end{split}
\end{equation}

Let us give upper bounds of $M_1(r)$ and $M_2(r)$ in this order. 
Suppose that $d(x,y) \le F_r(x,y)$. 
\begin{enumerate}
\item[(a)] 
Assume that $d_0(x)\le d_0(y)$. 
Then 
\[
d_0(y)\le d_0(x)+d(x,y)\le d_0(x)+c_* (r+d_0(y))^{1-\delta}\le d_0(x)+c_* \{(r+d_0(y))\vee 1\}
\]
and so 
\[
d_0(y)\le \frac{1}{1-c_*}(d_0(x)+c_* (r\vee 1)).
\]
This yields
\[
F_r(x,y)\lesssim (1+r+d_0(x))^{1-\delta}, \quad x,y\in E
\]
and 
\[
c(x,y)\lesssim (1+r+d_0(x))^p{\bf 1}_{d(x,y) \le 1}+(1+r+d_0(x))^q{\bf 1}_{d(x,y) > 1}, \quad x,y\in E.
\]
We also have by \eqref{eq:length-1},
\[
|\rho_r(x)-\rho_r(y)|
\le \frac{\delta d(x,y)}{(1+r+d_0(x))^{1-\delta}}, \quad x,y\in E.
\]
\item[(b)] 
Assume that $d_0(y)\le d_0(x)$. 
Then
\[
d_0(x)\le \frac{1}{1-c_*}(d_0(y)+c_* (r\vee 1))
\]
and so 
\[
1+r+d_0(y) \ge 1+r+(1-c_*)d_0(x)-c_*(r\vee 1) \ge (1-c_*)(1+r+d_0(x)).
\]
Hence by \eqref{eq:length-1},
\[
|\rho_r(x)-\rho_r(y)|
\lesssim \frac{d(x,y)}{(1+r+d_0(x))^{1-\delta}}, \quad x,y\in E.
\]
We also have 
\[
F_r(x,y)=c_* (r+d_0(x))^{1-\delta}, \quad x,y\in E
\]
and
\begin{equation*}
\begin{split}
c(x,y)
&\le  2(1+d_0(x))^p{\bf 1}_{d(x,y) \le 1}+2(1+d_0(x))^q{\bf 1}_{d(x,y) > 1}\\
&\lesssim (1+r+d_0(x))^p{\bf 1}_{d(x,y) \le 1}+(1+r+d_0(x))^q{\bf 1}_{d(x,y) > 1}, \quad x,y\in E.
\end{split}
\end{equation*}
\end{enumerate}

By (a) and (b) above, we get for some $c_{**}>0$, 
\begin{equation*}
\begin{split}
&\int_{0<d(x,y) \le F_r(x,y)}(\rho_r(x)-\rho_r(y))^2 J(x,y)\,m({\rm d}y)\\
&\lesssim \int_{0<d(x,y) \le c_{**} (1+r+d_0(x))^{1-\delta}} \frac{d(x,y)^2}{(1+r+d_0(x))^{2(1-\delta)}}
\frac{(1+r+d_0(x))^p}{d(x,y)^{\eta + \beta}}{\bf 1}_{d(x,y) \le 1}\,m({\rm d}y)\\
&+\int_{0<d(x,y) \le c_{**} (1+r+d_0(x))^{1-\delta}} \frac{ d(x,y)^2}{(1+r+d_0(x))^{2(1-\delta)}}
\frac{(1+r+d_0(x))^q}{d(x,y)^{\eta + \beta}}{\bf 1}_{d(x,y) > 1}\,m({\rm d}y)\\
&\lesssim (1+r+d_0(x))^{p-2(1-\delta)}
+ (1+r+d_0(x))^{q-2(1-\delta)}(1+r+d_0(x))^{(1-\delta)(2-\beta)}\\
&= (1+r+d_0(x))^{p-2(1-\delta)}
+ (1+r+d_0(x))^{q-\beta (1-\delta)},
\end{split}
\end{equation*}
which implies that 
\[
M_1(r)\lesssim r^{p-2(1-\delta)}+r^{q-\beta(1-\delta)}, \quad r\ge 1.
\]

Since 
\[
(1+r+d_0(y))^p\le (1+r+d_0(x)+d(x,y))^p\le 2^p\{(1+r+d_0(x))^p+d(x,y)^p\},
\]
we also obtain 
\begin{equation}\label{eq:ex-m2-1}
\begin{split}
M_2(r)
&\lesssim \int_{d(x,y) > c_* (r+d_0(x))^{1-\delta}}\frac{(1+r+d_0(x))^p+ d(x,y)^p}{d(x,y)^{\eta+\beta}}{\bf 1}_{d(x,y) \le 1}\,m({\rm d}y)\\
&+\int_{d(x,y) > c_* (r+d_0(x))^{1-\delta}}\frac{(1+r+d_0(x))^q+d(x,y)^q}{d(x,y)^{\eta+\beta}}{\bf 1}_{d(x,y) > 1}\,m({\rm d}y)\\
&\lesssim {\bf 1}_{c_* r^{1-\delta}<1}\left\{(1+r+d_0(x))^{p-\beta(1-\delta)}+(1+r+d_0(x))^{(p-\beta)(1-\delta)}\right\}\\
&+ (1+r+d_0(x))^{q-\beta(1-\delta)}+(1+r+d_0(x))^{(q-\beta)(1-\delta)}.
\end{split}
\end{equation}
In particular, if $r\ge c_*^{-1/(1-\delta)}$, then 
\begin{equation}\label{eq:ex-m2-2}
M_2(r)\lesssim r^{q-\beta(1-\delta)}+r^{(q-\beta)(1-\delta)} \asymp r^{q-\beta(1-\delta)}.
\end{equation}

By the definition of $\rho_r$, we have 
$m(K_{\rho_r}(R))\lesssim R^{\eta/\delta}$ for $R\ge 1$ and so $\mu_r=0$. 
Hence Theorem \ref{thm:ess} and \eqref{eq:ex-m2-2} yield 
$\lambda_e \lesssim   r^{q-\beta(1-\delta)}$ for any $r\ge c_*^{-1/(1-\delta)}$. 
Since $q-\beta(1-\delta)<0$, we have $\lambda_e=0$.

\item[(ii)]
We next prove that if $p=2$, then $\lambda_e<\infty$. 
Assume that $p=2$. 
For fixed constants $r>0$ and $c_*\in (0,1)$, let 
\[
\rho_r(x)=\log(r+d_0(x)), \quad x \in E
\]
and
\[
F_r(x,y)=c_* \{r+(d_0(x)\vee d_0(y))\}, \quad x,y\in E.
\]
Then as in \eqref{eq:length-1}, we have 
\begin{equation}\label{eq:length-2}
|\rho_r(x)-\rho_r(y)|
\le \frac{d(x,y)}{r+(d_0(x)\wedge d_0(y))}, \quad x,y\in E.
\end{equation}

Let us give upper bounds of $M_1(r)$ and $M_2(r)$ in this order. 
Suppose that $d(x,y) \le F_r(x,y)$. 
\begin{enumerate}
\item[(a)] 
Assume that $d_0(x)\le d_0(y)$. Then
\[
d_0(y)\le d_0(x)+d(x,y)\le d_0(x)+c_* (r+d_0(y))
\]
and so 
\[
d_0(y)\le \frac{1}{1-c_*}(d_0(x)+c_* r).
\]
Hence
\[
c(x,y)\lesssim (1+r+d_0(x))^2{\bf 1}_{d(x,y)\le 1}+ (1+r+d_0(x))^q{\bf 1}_{d(x,y)>1}, 
\quad x, y \in E\]
and 
\[
F_r(x,y)\lesssim r+d_0(x), \quad x, y \in E.
\]
By \eqref{eq:length-2}, we also have 
\[
|\rho_r(x)-\rho_r(y)| \le \frac{d(x,y)}{r+d_0(x)}, \quad x,y\in E.
\]

\item[(b)] 
Assume that $d_0(y)\le d_0(x)$. 
Then 
\[
c(x,y)\le 2(1+d_0(x))^2 {\bf 1}_{d(x,y) \le 1}+2(1+d_0(x))^q{\bf 1}_{d(x,y) > 1}
\] 
and $F_r(x,y)=c_* (r+d_0(x))$. 
The latter yields
\[
d_0(y)\ge d_0(x)-d(x,y) \ge d_0(x)-c_* (r+d_0(x)) = (1-c_*)d_0(x)-c_* r
\]
and so 
\[
r+d_0(y)\ge (1-c_*)(r+d_0(x)).
\]
Hence by \eqref{eq:length-2},
\[
|\rho_r(x)-\rho_r(y)| \le \frac{d(x,y)}{r+d_0(y)} 
\le \frac{1}{1-c_*} \cdot \frac{d(x,y)}{r+d_0(x)}, \quad x,y\in E.
\]
\end{enumerate}

By (a) and (b) above, we get for some $c_{**}>0$,
\begin{equation*}
\begin{split}
&\int_{0<d(x,y) \le F_r(x,y)}(\rho_r(x)-\rho_r(y))^2J(x,y)\,m({\rm d}y)\\
&\lesssim \int_{0<d(x,y) \le c_{**}(r+d_0(x))}\frac{d(x,y)^2}{(r+d_0(x))^2}\frac{(1+r+d_0(x))^2}{d(x,y)^{\eta+\beta}}{\bf 1}_{d(x,y) \le 1}\,m({\rm d}y)\\
&+\int_{0<d(x,y) \le c_{**}(r+d_0(x))}\frac{d(x,y)^2}{(r+d_0(x))^2}\frac{(1+r+d_0(x))^q}{d(x,y)^{\eta+\beta}}{\bf 1}_{d(x,y) > 1}\,m({\rm d}y)\\
&\lesssim \left(1+\frac{1}{r+d_0(x)}\right)^2+ \frac{(1+r+d_0(x))^q}{(r+d_0(x))^{\beta}},
\end{split}
\end{equation*}
where the last relation follows by the same calculation as in \eqref{eq:ibp-1}. 
Therefore, $M_1(r)\lesssim 1$ for $r\ge 1$.
By following the calculation in \eqref{eq:ex-m2-1} and \eqref{eq:ex-m2-2}, 
we also see that if $r\ge 1/c_*$, then 
\[
M_2(r)\lesssim r^{q-\beta}.
\]

On the other hand, 
we see by the definition of $\rho_r$ that  
$m(K_{\rho_r}(R))\lesssim e^{\eta R}$ for $R\ge 1$ and so $\mu_r\le \eta$. 
Then Theorem \ref{thm:ess} implies that for any $r\ge 1/c_*$,
\[
\lambda_e \lesssim \frac{\eta^2}{2} + r^{q-\beta}.
\]
Since $q < \beta$, we obtain $\lambda_e \lesssim \eta^2/2$.
\end{enumerate}

\begin{rem}\rm 
We here note the sharpness of Theorem \ref{thm:ess}. 
Let ${\rm d}x$ be the Lebesgue measure on ${\mathbb R}^d$. 
For $p\ge 0$ and $q\ge 0$, define 
\[
c(x,y)=\{(1+|x|)^p+(1+|y|)^p\}{\bf 1}_{|x-y| \le 1}+\{(1+|x|)^q+(1+|y|)^q\}{\bf 1}_{|x-y| > 1}, \quad x,y\in {\mathbb R}^d.
\]
For $\alpha\in (0,2)$, let $J(x,y)$ be a positive measurable function on ${\mathbb R}^d\times {\mathbb R}^d\setminus {\rm diag}$ such that 
\[
J(x,y)\asymp \frac{c(x,y)}{|x-y|^{d+\alpha}}, \quad (x,y)\in {\mathbb R}^d\times {\mathbb R}^d \setminus {\rm diag}.
\]
Assume that $q\in [0,\alpha)$. 
If we define 
\begin{equation*}
\begin{split}
{\cal D}({\cal E})&=\left\{u\in L^2({\mathbb R}^d;{\rm d}x) 
\mid \iint_{{\mathbb R}^d\times {\mathbb R}^d \setminus {\rm diag}}(u(x)-u(y))^2J(x,y)\,{\rm d}x{\rm d}y<\infty\right\},\\
{\cal E}(u,u)&=\iint_{{\mathbb R}^d\times {\mathbb R}^d \setminus {\rm diag}}(u(x)-u(y))^2J(x,y)\,{\rm d}x{\rm d}y, \quad u\in {\cal D}({\cal E}), 
\end{split}
\end{equation*}
then $C_0^{\infty}({\mathbb R}^d)$ is dense in ${\cal D}({\cal E})$ 
with respect to the norm $\|u\|_{{\cal E}}=({\cal E}(u,u)+\|u\|_{L^2({\mathbb R}^d;{\rm d}x)}^2)^{1/2}$. 
Hence, if we let ${\cal F}$ be the closure of $C_0^{\infty}({\mathbb R}^d)$ 
with respect to the norm $\|\cdot\|_{{\cal E}}$, 
then $({\cal E},{\cal F})$ is a regular Dirichlet form on $L^2({\mathbb R}^d;{\rm d}x)$. 

We know from \cite[Theorem 1.1]{SW23} that 
$\lambda_e<\infty$ if and only if $p \le 2$, independently of the value of $q\in [0,\alpha)$.
We further see by (i) above that $\lambda_e=0$ if $0\le p <2$. 
Let us assume that $p=2$. 
According to the calculations in \cite[Proposition 2.6 and Lemma 2.8]{SW23},  
there exist positive constants $R_0$ and $C_0$ such that 
for any $u\in C_0^{\infty}({\mathbb R}^d)$ satisfying $u=0$ on $K_0:=\{x\in {\mathbb R}^d \mid |x|\le R_0\}$, 
\begin{equation}\label{eq:outside}
{\cal E}(u,u)\ge C_0\int_{|x|>R_0} u(x)^2\,{\rm d}x.
\end{equation}
On the other hand, 
by \cite[Theorem 3.2, Remark 3.3 (c) and Corollary 4.3]{LS19}, 
Persson's formula is applicable to $({\cal E},{\cal F})$: 
\[
\lambda_e
=\sup_{K\subset {\mathbb R}^d: \text{compact}}
\inf\left\{{\cal E}(u,u) \mid u\in C_0^{\infty}({\mathbb R}^d\setminus K), \ \|u\|_{L^2({\mathbb R}^d;{\rm d}x)}=1\right\}.
\]
Combining this with \eqref{eq:outside}, we have 
\[
\lambda_e
\ge 
\inf\left\{{\cal E}(u,u) \mid u\in C_0^{\infty}({\mathbb R}^d\setminus K_0), \ \|u\|_{L^2({\mathbb R}^d;{\rm d}x)}=1\right\}
\ge C_0>0.
\]
Namely, Theorem \ref{thm:ess} is sharp in regard to the positivity of $\lambda_e$.

\end{rem}

\subsection{Time change}

In this subsection, we discuss how the coefficient in the underlying measure affects the upper bound of $\lambda_e$. 
We impose the next conditions on  the volume growth and jump kernel.
\begin{itemize}
\item 
There exist positive constants $\eta$, $C_1$ and $C_2$ such that 
\[
C_1 r^{\eta} \le m(K_x(r)) \le  C_2 r^{\eta}, \quad x\in E, \ r>0.
\]
\item
There exist positive constants $C_3$ and $\beta\in (0,2)$ such that 
\[
J(x,y) \le \frac{C_3}{d(x,y)^{\eta+\beta}}, \quad x,y\in E\times E \setminus {\rm diag}.
\]
\end{itemize}

Let $w(x)$ be a positive Borel measurable function on ${\mathbb R}^d$ such that for some $p>0$, 
\[
w(x)\asymp (1+d_0(x))^p, \quad x\in E.
\]
We define a measure $\mu$ on $E$ by $\mu({\rm d}x)=w(x)^{-1}\,m({\rm d}x)$. 
Let $({\cal F}_{e},{\cal E})$ be the extended Dirichlet space of $({\cal F},{\cal E})$ 
(see \cite[p.~41]{FOT11} for definition). 
Let $(\check{{\cal E}},\check{\cal F})$ be the time changed Dirichlet form of $({\cal E},{\cal F})$ on $E$  
with respect to the measure $\mu$ (see \cite[(6.2.4)]{FOT11} for definition). 
Since $\mu$ is of full support, we know that 
\begin{equation}\label{eq:time-change}
\check{\cal F}={\cal F}_e\cap L^2(E;\mu), \quad \check{\cal E}(u,u)={\cal E}(u,u), \quad u\in \check{\cal F}
\end{equation}
(see \cite[(6.2.22)]{FOT11}). 
In particular, $(\check{{\cal E}},\check{\cal F})$ is a regular Dirichlet form on $L^2(E;\mu)$ 
with core ${\cal F}\cap C_0(E)$ (\cite[Theorem 6.2.1 (iii)]{FOT11}).
If we define 
\[
\check{J}(x,{\rm d}y)=w(x)J(x,y)\,m({\rm d}y), 
\]
then 
\[
J(x,y)\,m({\rm d}y)m({\rm d}x)=\check{J}(x,{\rm d}y)\mu({\rm d}x).
\]
Moreover, we see by \eqref{eq:time-change} that for any $u\in {\cal F}\cap C_0(E)$, 
\begin{equation*}
\begin{split}
\check{\cal E}(u,u)={\cal E}(u,u)
&=\iint_{E\times E \setminus {\rm diag}}(u(x)-u(y))^2\, J(x,y)\,m({\rm d}y)m({\rm d}x)\\
&=\iint_{E\times E \setminus {\rm diag}}(u(x)-u(y))^2\,\check{J}(x,{\rm d}y)\mu({\rm d}x).
\end{split}
\end{equation*}

\begin{enumerate}
\item[(i)] Assume first that $p<\beta$. 
Let $\rho_r(x)=(r+d_0(x))^\delta$ and 
\[
F_r(x,y)=c\{r+(d_0(x)\vee d_0(y))\}^{1-\delta}, \quad x,y\in E
\]
for some $c\in (0,1)$ and $\delta\in (0,1)$ with $p<\beta(1-\delta)$. 
Then as in Subsection \ref{subsect:coeff} (i), we have 
\[
M_1(r)\lesssim 
r^{p-\beta(1-\delta)}, \quad 
M_2(r)\lesssim r^{p-\beta(1-\delta)}, \quad r\ge 1.
\]

\begin{enumerate}
\item[(a)] 
Assume that $p\le \eta$. Then for all sufficiently large $R>1$,
\[
\mu(K_{\rho_r}(R))
\lesssim 
\begin{dcases}
R^{(\eta-p)/\delta} & (p<\eta),\\
\log R & (p=\eta).
\end{dcases}
\]
This yields $\mu_r=0$ and so $\lambda_e=0$ by Theorem \ref{thm:ess}. 
\item[(b)] 
Assume that $p>\eta$ and so $\mu(E)<\infty$. 
Since $\eta<\beta$ by assumption, 
we can show that $({\cal E},{\cal F})$ is recurrent 
by \cite[Theorem A.3]{S23} (see also references therein for previous results) 
and by following the calculation in \cite[Example A.5]{S23}. 
Hence by \cite[Theorem 1.6.3]{FOT11}, it follows that $1\in {\cal F}_e$ and ${\cal E}(1,1)=0$. 
Since this and \eqref{eq:time-change} yield $1\in \check{\cal F}$, 
Theorem \ref{thm:ess-f} is applicable to $(\check{{\cal E}},\check{\cal F})$. 

Assume in addition that $C_2<C_1p(p-\eta)^{-1}$. 
Then for all sufficiently large $R>1$, 
\[
\mu(K_{\rho_r}(R)^c)
\gtrsim R^{-(p-\eta)/\delta}.
\]
This implies that $\nu_r=0$ and thus $\lambda_e=0$ by Theorem \ref{thm:ess-f}. 
\end{enumerate}

\item[(ii)] 
Assume next that  $p=\beta$. 
Let $\rho_r(x)=\log (r+d_0(x))$ and 
\[
F_r(x,y)=c\{r+(d_0(x)\vee d_0(y))\}, \quad x,y\in E
\]
for some $c\in (0,1)$.  
Then as in Subsection \ref{subsect:coeff} (ii), 
\[
M_1(r)\lesssim 1/r^{\beta}, \quad 
M_2(r)\lesssim 1, \quad r\ge 1.
\]
\begin{enumerate}
\item[(a)]
Assume that  $p\le \eta$. 
Then for all sufficiently large $R>1$,
\[
\mu(K_{\rho_r}(R))
\lesssim   
\begin{dcases}
e^{(\eta-p)R}, & p<\eta,\\
R, & p=\eta.
\end{dcases}
\]
Hence Theorem \ref{thm:ess} yields $\lambda_e<\infty$ for $p<\eta$, 
and $\lambda_e=0$ for $p=\eta$. 

\item[(b)] 
Assume that $p>\eta$ and so $\mu(E)<\infty$.  
By the same argument as in (i)(b), 
we can apply Theorem \ref{thm:ess-f} to  $(\check{{\cal E}},\check{\cal F})$. 

Assume in addition that $C_2 < C_1 p(p-\eta)^{-1}$. 
Then for all sufficiently large $R\ge 1$, 
\[
\mu(K_{\rho_r}(R)^c) \gtrsim  e^{-(p-\eta)R}. 
\]
Therefore, Theorem \ref{thm:ess-f} yields $\lambda_e<\infty$. 
\end{enumerate}
\end{enumerate}

We now examine the sharpness of the upper bound of $\lambda_e$ for time changed Dirichlet forms. 
In what follows, we assume that $E={\mathbb R}^d$. 
Let ${\rm d}x$ be the $d$-dimensional Lebesgue measure. 
We also assume that $m({\rm d}x)={\rm d}x$ and 
\[
J(x,y)\asymp |x-y|^{-(d+\alpha)}, \quad 
(x,y)\in {\mathbb R}^d \times {\mathbb R}^d\setminus {\rm diag}
\] 
for some $\alpha\in (0,2)$. 
Then $({\cal E},{\cal F})$ is a regular Dirichlet form on $L^2({\mathbb R}^d;{\rm d}x)$ with $C_0^{\infty}({\mathbb R}^d)$ as a core. 

Let $w(x)$ be a positive measurable function on ${\mathbb R}^d$ 
such that $w(x)\asymp (1+|x|)^p \ (x\in {\mathbb R}^d)$ for some $p>0$. 
We then define the measure $\mu$ on ${\mathbb R}^d$ by $\mu({\rm d}x)=w(x)^{-1}{\rm d}x$. 
Let $(\check{\cal E},\check{\cal F})$ be a time changed Dirichlet form of $({\cal E},{\cal F})$ 
with respect to $\mu$. 
By \cite{CW14, HW25, M21, W19, WZ21}, 
we already know necessary and sufficient condition for noncompactness of 
the Markovian semigroups associated with $(\check{\cal E},\check{\cal F})$. 
In particular, these conditions are consistent with those in (i) and (ii) of this section. 
In what follows, we discuss the positivity of $\lambda_e$ for $d>\alpha$. 

For $\delta>0$, let $\phi_{\delta}(x)=(1+|x|^2)^{-\delta}$ and 
\[
{\cal A}\phi_{\delta}(x)=\int_{|z|>1}(\phi_{\delta}(x+z)-\phi_{\delta}(x))|z|^{-(d+\alpha)}\,{\rm d}z.
\]
For $R>0$, define $K_0(R)=\{x\in {\mathbb R}^d \mid |x|\le R\}$. 
The next proposition states that Theorem \ref{thm:ess} is sharp in regard to the positivity of $\lambda_e$.
\begin{prop}
\begin{enumerate}
\item[{\rm (1)}] 
Let $\delta>0$ and  $g\in C_0^{\infty}({\mathbb R}^d)$.  Then 
\[
\iint_{|z|>1}|\phi_{\delta}(x+z)-\phi_{\delta}(x)| |g(x+z)-g(x)| |z|^{-(d+\alpha)} \,{\rm d}z{\rm d}x<\infty
\]
and 
\[
\iint_{|z|>1}(\phi_{\delta}(x+z)-\phi_{\delta}(x))(g(x+z)-g(x)) |z|^{-(d+\alpha)} \,{\rm d}z{\rm d}x
=-2\int_{{\mathbb R}^d}{\cal A}\phi_{\delta}(x)g(x)\,{\rm d}x.
\]

\item[{\rm (2)}]
Suppose that $d>\alpha$ and $p=\alpha$. 
Then there exist positive constants $C$ and $R_0$ such that for any $g\in C_0^{\infty}({\mathbb R}^d\setminus K_0(R_0))$, 
\[
\iint_{|z|>1}(g(x+z)-g(x))^2 |z|^{-(d+\alpha)} \,{\rm d}z{\rm d}x 
\ge C \int_{|x|>R_0}g(x)^2\,\mu({\rm d}x).
\]
\item[{\rm (3)}]
Suppose that $d>\alpha$ and $p=\alpha$. 
Let $\lambda_e$ be the bottom of the essential spectrum of the nonpositive self-adjoint operator on $L^2({\mathbb R}^d;\mu)$
associated with $(\check{\cal E},\check{\cal F})$ as in \eqref{eq:generator}. 
Then $0<\lambda_e<\infty$.
\end{enumerate}
\end{prop}

\begin{proof}
We first prove (1). 
Let $g\in C_0^{\infty}({\mathbb R}^d)$, and let $\omega_d$ be a surface area of the unit ball in ${\mathbb R}^d$. 
Then 
\begin{equation*}
\begin{split}
&\iint_{|z|>1}|\phi_{\delta}(x+z)-\phi_{\delta}(x)| |g(x)| |z|^{-(d+\alpha)} \,{\rm d}z{\rm d}x \\
&\le 2\int_{{\mathbb R}^d}|g(x)|\,{\rm d}x \int_{|z|>1}|z|^{-(d+\alpha)}\,{\rm d}z
=\frac{d\omega_d}{\alpha}\int_{{\mathbb R}^d}|g(x)|\,{\rm d}x<\infty
\end{split}
\end{equation*}
and 
\begin{equation*}
\begin{split}
&\iint_{|z|>1}|\phi_{\delta}(x+z)-\phi_{\delta}(x)| |g(x+z)| |z|^{-(d+\alpha)} \,{\rm d}z{\rm d}x\\
&\le 2 \int_{|z|>1} \left(\int_{{\mathbb R}^d}|g(x+z)| \,{\rm d}x \right)|z|^{-(d+\alpha)} \,{\rm d}z 
= 2\int_{{\mathbb R}^d}|g(x)|\,{\rm d}x \int_{|z|>1}|z|^{-(d+\alpha)}\,{\rm d}z<\infty.
\end{split}
\end{equation*}
We thus arrive at the first assertion. 

By the first assertion, we have 
\[
-\iint_{|z|>1}(\phi_{\delta}(x+z)-\phi_{\delta}(x)) g(x) |z|^{-(d+\alpha)} \,{\rm d}z{\rm d}x
=-\int_{{\mathbb R}^d}{\cal A}\phi_{\delta}(x)g(x)\,{\rm d}x.
\]
By the translation invariance of the Lebesgue measure and the symmetry of the function $|z|^{-(d+\alpha)}$, 
we also obtain 
\begin{equation*}
\begin{split}
&\iint_{|z|>1}(\phi_{\delta}(x+z)-\phi_{\delta}(x)) g(x+z) |z|^{-(d+\alpha)} \,{\rm d}z{\rm d}x\\
&=\iint_{|z|>1}(\phi_{\delta}(x)-\phi_{\delta}(x-z)) g(x) |z|^{-(d+\alpha)} \,{\rm d}z{\rm d}x\\
&=\iint_{|z|>1}(\phi_{\delta}(x)-\phi_{\delta}(x+z)) g(x) |z|^{-(d+\alpha)} \,{\rm d}z{\rm d}x
=-\int_{{\mathbb R}^d}{\cal A}\phi_{\delta}(x)g(x)\,{\rm d}x.
\end{split}
\end{equation*}
Hence we get the second assertion.

We next prove (2). 
It follows by \cite[Proposition 3.11]{S23} that if $d>\alpha$, 
then there exist positive constants $\delta_0$, $C_0$ and $R_0$ such that 
\[
\frac{-{\cal A}\phi_{\delta_0}}{\phi_{\delta_0}}(x)\ge \frac{C_0}{(1+|x|)^{\alpha}}, \quad |x|>R_0.
\]
Then by (1) and the proof of \cite[Lemma 2.8]{SW23}, we have 
for any $g\in C_0^{\infty}({\mathbb R}^d\setminus K_0(R_0))$, 
\begin{equation*}
\begin{split}
&\iint_{|z|>1}(g(x+z)-g(x))^2 |z|^{-(d+\alpha)} \,{\rm d}z{\rm d}x 
\ge 2 \int_{|x|>R_0}\frac{-{\cal A}\phi_{\delta_0}}{\phi_{\delta_0}}(x)g(x)^2\,{\rm d}x \\
&\ge 2C_0 \int_{|x|>R_0}g(x)^2\frac{1}{(1+|x|)^{\alpha}}\,{\rm d}x
\ge c_1 \int_{|x|>R_0}g(x)^2 \,\mu({\rm d}x).
\end{split}
\end{equation*}
Therefore, the proof of (2) is complete. 

We finally prove (3). 
Let $\{\check{p}_t\}_{t\ge 0}$ be the Markovian semigroup on $L^2({\mathbb R}^d;\mu)$ 
associated with $(\check{\cal E},\check{\cal F})$. 
Since $\{\check{p}_t\}_{t\ge 0}$ is ultracontractive by \cite[Lemma 4.9]{M21}, 
we obtain by \cite[Remark 2.2 (c), Theorem 3.2 and Remark 3.3 (c)]{LS19} together with \eqref{eq:time-change}, 
\begin{equation}\label{eq:ess-time}
\lambda_e
=\sup_{K\subset {\mathbb R}^d: \text{compact}}
\inf\left\{{\cal E}(u,u) \mid u\in C_0^{\infty}({\mathbb R}^d\setminus K), \ \|u\|_{L^2({\mathbb R}^d;\mu)}=1\right\}.
\end{equation}
On the other hand, we see by (2) that there exist positive constants $c_2$, $c_3$ and $R_0$ such that 
for any $g\in C_0^{\infty}({\mathbb R}^d\setminus K_0(R_0))$, 
\[
{\cal E}(g,g)\ge c_2  \iint_{|z|>1}(g(x+z)-g(x))^2 |z|^{-(d+\alpha)} \,{\rm d}z{\rm d}x  
\ge c_3\int_{|x|>R_0}g(x)^2 \,\mu({\rm d}x).
\]
Combining this with \eqref{eq:ess-time}, we have 
\[\lambda_e 
\ge \inf\left\{{\cal E}(u,u) \mid u\in C_0^{\infty}({\mathbb R}^d\setminus K_0(R_0)), \ \|u\|_{L^2({\mathbb R}^d;\mu)}=1\right\}
\ge c_3>0. 
\]
Since $\lambda_e<\infty$ by (ii) (a) in this section, we complete the proof of (3).
\end{proof}

\begin{rem}\rm
When $\alpha>d=1$ and $p=\alpha$, 
it is unclear if Persson's formula (\cite[Theorem 5.5]{BGS23}, \cite[Theorem 3.2]{LS19}) is applicable to 
$(\check{\cal E}, \check{\cal F})$ at this moment. 
In particular, we do not know the positivity of $\lambda_e$. 
\end{rem}

\section{Ornstein-Uhlenbeck type operators}
\label{sect:ou-type}

In this section, we discuss  the bottom of the spectrum of a non-local operator 
which is related to the fractional Laplacian with the drift of Ornstein-Uhlenbeck type 
(see, e.g., the introduction of \cite{WW15} and references therein for details on this operator). 

Let $E$ be a locally compact separable metric space, 
and let $m$ be a positive Radon measure on $E$ with full support. 
Recall that $K_x(r)=\left\{y\in E \mid d(x,y)\le r\right\}$ for $x\in E$ and $r>0$. 
We assume that $K_x(r)$ is compact for any $x\in E$ and $r>0$. 
We also assume that for some positive constants $C_1$, $C_2$ and $\eta$,  
\begin{equation}\label{eq:ou-volume}
C_1 r^{\eta} \le m(K_x(r)) \le C_2r^{\eta}, \quad x\in E, \ r>0.
\end{equation}
Let $J(x,y)$ be a positive measurable function on $E\times E$ such that 
for some positive constants $C_3$ and $\beta\in (0,2)$, 
\begin{equation}\label{eq:ou-jump}
J(x,y) \le \frac{C_3}{d(x,y)^{\eta+\beta}}, \quad (x,y)\in E\times E \setminus {\rm diag}.
\end{equation}

Let $V(r)$ be a positive increasing function on $[0,\infty)$, 
and let $\mu_V({\rm d}x)=e^{-V(d_0(x))}\,m({\rm d}x)$. 
We define the quadratic form $({\cal E},{\cal D}({\cal E}))$ on $L^2(E;\mu_V)$ by 
\begin{equation*}
\begin{split}
{\cal D}({\cal E})
&=\left\{u\in L^2(E;\mu_V) \mid \iint_{E\times E \setminus{\rm diag}}(u(x)-u(y))^2 \, J(x,y) m({\rm d}y)\mu_V({\rm d}x)<\infty\right\}, \\
{\cal E}(u,u)
&=\iint_{E\times E\setminus {\rm diag}}(u(x)-u(y))^2 \, J(x,y) m({\rm d}y)\mu_V({\rm d}x), \quad u\in {\cal D}({\cal E}).
\end{split}
\end{equation*}
For $u\in {\cal D}({\cal E})$, let 
\[{\cal E}_1(u,u)={\cal E}(u,u)+\|u\|_{L^2(E;m)}^2, \quad \|u\|_{{\cal E}}=\sqrt{{\cal E}_1(u,u)}.
\]
Then ${\cal D}({\cal E})$ is a Banach space with the norm $\|\cdot \|_{{\cal E}}$.  
We also note that, 
if $C_0^{{\rm lip}}(E)$ denotes the totality of Lipschitz continuous functions on $E$ with compact support, 
then $C_0^{{\rm lip}}(E)\subset {\cal D}({\cal E})$. 
Hence if ${\cal F}$ is the $\|\cdot\|_{{\cal E}}$-closure of $C_0^{{\rm lip}}(E)$, 
then $({\cal E},{\cal F})$ is a regular Dirichlet form on $L^2(E;m)$.

For some $o\in E$, let $d_0(x)=d(o,x) \ (x\in E)$. 
Let $C_b^{{\rm lip}}(E)$ be the totality of bounded Lipschitz continuous functions on $E$. 
To apply Theorem \ref{thm:ess-f} for $({\cal E},{\cal F})$, we prove
\begin{lem}\label{lem:bdd-lip}
\begin{enumerate}
\item[{\rm (i)}] The function $d_0$ belongs to ${\cal F}_{{\rm loc}}$.
\item[{\rm (ii)}] 
Suppose that $\mu_V(E)<\infty$. 
Then $C_b^{{\rm lip}}(E)\subset {\cal D}({\cal E})$,  and ${\cal F}$ coincides with the $\|\cdot\|_{{\cal E}}$-closure of $C_b^{{\rm lip}}(E)$.
In particular, $({\cal E},{\cal F})$ is recurrent. 
\end{enumerate}
\end{lem}

\begin{proof}
We first prove (i). For $n\in {\mathbb N}$, we define 
\[
w_n(t)=
\begin{dcases}
t, & t\le n, \\
n+1-t, & n< t\le n+1, \\
0, & t>n+1.
\end{dcases}
\]
Then there exists a positive constant $L$ such that for any $n\in {\mathbb N}$, 
\[
|w_n(t)-w_n(s)|\le L |t-s|, \ s,t\in {\mathbb R}.
\]
Let $\varphi_n(x)=w_n(d_0(x)) \ (x\in E)$. Then for any $n\in {\mathbb N}$ and $x,y\in E$,
\begin{equation}\label{eq:ou-lip}
|\varphi_n(x) - \varphi_n(y)|=|w_n(d_0(x))-w_n(d_0(y))| 
\le L|d_0(x)-d_0(y)| \le L d(x,y),
\end{equation}
which yields $\varphi_n \in C_0^{{\rm lip}}(E)$. 
Since $d_0(x)=\varphi_n(x)$ for any $x\in E$ with $d_0(x)\le n$, 
we have (i).

We next prove the first assertion of (ii). Assume that $\mu_V(E)<\infty$. 
Then for $u\in C_b^{{\rm lip}}(E)$, 
\begin{equation*}
\begin{split}
&\iint_{E\times E\setminus {\rm diag}}(u(x)-u(y))^2 \, J(x,y) m({\rm d}y)\mu_V({\rm d}x)\\
&=\iint_{0<d(x,y)\le 1}(u(x)-u(y))^2 \, J(x,y) m({\rm d}y)\mu_V({\rm d}x)\\
&+\iint_{d(x,y)>1}(u(x)-u(y))^2 \, J(x,y) m({\rm d}y)\mu_V({\rm d}x)\\
&\le c_1 \left(\iint_{0<d(x,y)\le 1}d(x,y)^2 \, J(x,y) m({\rm d}y)\mu_V({\rm d}x)
+\iint_{d(x,y)>1} J(x,y) m({\rm d}y)\mu_V({\rm d}x)\right).
\end{split}
\end{equation*}
By \eqref{eq:ou-jump} and the same calculation as \eqref{eq:ibp-1} and \eqref{eq:ibp-2}, we have 
\begin{equation*}
\begin{split}
\iint_{0<d(x,y)\le 1}d(x,y)^2 \, J(x,y) m({\rm d}y)\mu_V({\rm d}x)
&\le C_3\int_E \left( \int_{0<d(x,y)\le 1} \frac{d(x,y)^2}{d(x,y)^{\eta+\beta}}\,m({\rm d}y)\right)\,\mu_V({\rm d}x)\\
&\le c_2
\end{split}
\end{equation*}
and 
\[
\iint_{d(x,y)>1} J(x,y) m({\rm d}y)\mu_V({\rm d}x)
\le C_3\int_E \left(\int_{d(x,y)>1} \frac{1}{d(x,y)^{\eta+\beta}}\,m({\rm d}y)\right)\,\mu_V({\rm d}x)\le c_3.
\]
We thus have $u\in  {\cal D}({\cal E})$ and so 
$C_b^{{\rm lip}}(E)\subset {\cal D}({\cal E})$.

We finally prove the second assertion of (ii). 
In order to do so, it is sufficient to show that $C_b^{{\rm lip}}(E)\subset {\cal F}$. 
Let $u\in C_b^{{\rm lip}}(E)$. 
Then $u\varphi_n \in C_0^{{\rm lip}}(E)$. 
In particular, if we let $\psi_n(x)=1-\varphi_n(x) \ (x\in E)$, then 
\begin{equation*}
\begin{split}
&{\cal E}_1(u-u\varphi_n, u-u\varphi_n)
={\cal E}_1(u\psi_n, u\psi_n)\\
&=\iint_{E\times E \setminus {\rm diag}}(u(x)\psi_n(x)-u(y)\psi_n(y))^2 \, J(x,y) m({\rm d}y)\mu_V({\rm d}x)
+\int_E u(x)^2\psi_n(x)^2\,\mu_V({\rm d}x).
\end{split}
\end{equation*}
Since \eqref{eq:ou-lip} holds and $u$ is bounded and Lipschitz continuous, 
there exists $M>0$ such that 
for any $n\in {\mathbb N}$ and $x,y\in E$, 
\begin{equation*}
\begin{split}
\left(u(x)\psi_n(x)-u(y)\psi_n(y)\right)^2
&=\left\{u(x)(\psi_n(x)-\psi_n(y))+\psi_n(y)(u(x)-u(y))\right\}^2\\
&\le M(1\wedge d(x,y)^2).
\end{split}
\end{equation*}
By noting that $\lim_{n\rightarrow\infty}\psi_n(x)=0$ for any $x\in E$ and 
\[
\iint_{E\times E \setminus {\rm diag}} (1\wedge d(x,y)^2)\,J(x,y)\,m({\rm d}y)\mu_V({\rm d}x)<\infty,
\]
the dominated convergent theorem yields 
\[
\lim_{n\rightarrow\infty}{\cal E}_1(u-u\varphi_n, u-u\varphi_n)=0
\]
and so $u \in {\cal F}$.
Hence the proof is complete. 
\end{proof}

Let us discuss the upper bound of $\lambda_e$ for $({\cal E},{\cal F})$ under the condition that $\mu_V(E)<\infty$. 
Note that $({\cal E},{\cal F})$ is recurrent by Lemma \ref{lem:bdd-lip}, 
and so Theorem \ref{thm:ess-f} is applicable to $({\cal E},{\cal F})$.  
Define 
\[
J(x,{\rm d}y)=\frac{1}{2}(1+e^{V(d_0(x))-V(d_0(y))})J(x,y)\,m({\rm d}y), \quad 
J({\rm d}x,{\rm d}y)=J(x,{\rm d}y)\mu_V({\rm d}x). 
\]
Then $J(A\times B)=J(B\times A)$ for any $A,B\in {\cal B}(E)$, and for any $u\in C_b^{{\rm lip}}(E)$,
\[
{\cal E}(u,u)
=\frac{1}{2}\iint_{E \times E}(u(x)-u(y))^2(1+e^{V(d_0(x))-V(d_0(y))})J(x,y)\,m({\rm d}y)\mu_V({\rm d}x).
\]

In what follows, we assume that for some positive constants $\delta$ and $C_2$, 
\begin{equation}\label{eq:ratio}
\frac{e^{V(r)}}{e^{V(s)}}\le C_2\left(\frac{r}{s}\right)^{\delta}, \quad 0<s<r<\infty.
\end{equation}
We take $\rho_r(x)=d_0(x)$ and $F_r(x,y)=r$ for $r>0$. We first give an upper bound of $M_1(r)$. 
By assumption, 
\begin{equation*}
\begin{split}
\int_{0<d(x,y) \le r} d(x,y)^2 J(x,{\rm d}y)
&\lesssim \int_{0<d(x,y) \le r}d(x,y)^{2-( \eta + \beta)}(1+e^{V(d_0(x))-V(d_0(y))})\,m({\rm d}y)\\
&\asymp r^{2-\beta} + \int_{0<d(x,y) \le r} d(x,y)^{2-( \eta + \beta)} e^{V(d_0(x))-V(d_0(y))}\,m({\rm d}y).
\end{split}
\end{equation*}
Since the function $V(r)$ is increasing, we have by \eqref{eq:ratio}, 
\begin{equation}\label{eq:ou-m1-1}
\begin{split}
&\int_{0<d(x,y)\le r, \, d_0(y)>d_0(x)/2} d(x,y)^{2-( \eta + \beta)} e^{V(d_0(x))-V(d_0(y))}\,m({\rm d}y)\\
&\le e^{V(d_0(x))-V(d_0(x)/2)} \int_{0<d(x,y) \le r, \, d_0(y)>d_0(x)/2} d(x,y)^{2-( \eta + \beta)} \,m({\rm d}y)\\
&\lesssim \int_{0<d(x,y) \le r} d(x,y)^{2-( \eta + \beta)} \,m({\rm d}y) 
\asymp  r^{2-\beta}.
\end{split}
\end{equation}
On the other hand, if $0<d(x,y)\le r$ and $d_0(y)\le d_0(x)/2$, then 
\[
r\ge d(x,y)\ge d_0(x)-d_0(y)\ge \frac{d_0(x)}{2}
\]
and so $d_0(x)\le 2r$. Therefore, by \eqref{eq:ratio},
\begin{equation*}
\begin{split}
&\int_{0<d(x,y)\le r, \, d_0(y) \le d_0(x)/2} d(x,y)^{2-( \eta + \beta)} e^{V(d_0(x))-V(d_0(y))}\,m({\rm d}y)\\
&\le e^{V(2r)}\int_{0<d(x,y) \le r}  d(x,y)^{2-( \eta + \beta)} \,m({\rm d}y)
\asymp r^{2-\beta}e^{V(r)}.
\end{split}
\end{equation*}
Combining this with \eqref{eq:ou-m1-1}, we get 
\begin{equation}\label{eq:ou-m1}
M_1(r)\lesssim r^{2-\beta}e^{V(r)}.
\end{equation}

We next give an upper bound of $M_2(r)$. 
By assumption, 
\begin{equation*}
\begin{split}
\int_{d(x,y)> r} J(x,{\rm d}y)
&\lesssim \int_{d(x,y) > r}\frac{1}{d(x,y)^{\eta + \beta}}(1+e^{V(d_0(x))-V(d_0(y))})\,m({\rm d}y)\\
&\asymp r^{-\beta} + \int_{d(x,y) > r} \frac{1}{d(x,y)^{\eta + \beta}} e^{V(d_0(x))-V(d_0(y))}\,m({\rm d}y).
\end{split}
\end{equation*}
Since
\begin{equation}\label{eq:bound-v}
\int_{d(x,y) > r}\frac{1}{d(x,y)^{\eta + \beta}}e^{-V(d_0(y))}\,m({\rm d}y)
\lesssim 
r^{-(\eta + \beta)},
\end{equation}
there exists $c_1>0$ such that for any $x\in E$ with $d_0(x)\le r$,
\begin{equation*}
\begin{split}
\int_{d(x,y) >r, \, d_0(y) \le d_0(x)/2} \frac{1}{d(x,y)^{ \eta + \beta }} e^{V(d_0(x))-V(d_0(y))}\,m({\rm d}y) 
&\le e^{V(r)}\int_{d(x,y) > r }\frac{1}{d(x,y)^{\eta + \beta}}e^{-V(d_0(y))}\,m({\rm d}y)\\
&\le c_1 
r^{-(\eta + \beta)}e^{V(r)}.
\end{split}
\end{equation*}
We note that, if $d_0(y) \le d_0(x)/2$, 
then $d(x,y)\ge d_0(x)/2$ by the triangle inequality.
Hence by \eqref{eq:bound-v}, there exists $c_2>0$ such that 
for any $x\in E$ with $d_0(x)>r$, 
\begin{equation*}
\begin{split}
&\int_{d(x,y) >r, \, d_0(y) \le d_0(x)/2} \frac{1}{d(x,y)^{ \eta + \beta }} e^{V(d_0(x))-V(d_0(y))}\,m({\rm d}y) \\  
&\le e^{V(d_0(x))}\int_{d(x,y) \ge d_0(x)/2}\frac{1}{d(x,y)^{\eta + \beta}}e^{-V(d_0(y))}\,m({\rm d}y) 
\le c_2
d_0(x)^{-(\eta + \beta)}e^{V(d_0(x))}.
\end{split}
\end{equation*}
By \eqref{eq:ratio}, we also have 
\begin{equation*}
\begin{split}
\int_{d(x,y) > r, \, d_0(y) > d_0(x)/2} \frac{1}{d(x,y)^{ \eta + \beta }} e^{V(d_0(x))-V(d_0(y))}\,m({\rm d}y) 
&\le \frac{e^{V(d_0(x))}}{e^{V(d_0(x)/2)}}\int_{d(x,y) > r}\frac{1}{d(x,y)^{\eta + \beta}}\,m({\rm d}y)\\
&\lesssim r^{-\beta}.
\end{split}
\end{equation*}
Therefore, the argument above yields for $r\ge 1$,
\begin{equation}\label{eq:ou-m2}
M_2(r)\lesssim
r^{-\beta}+\sup_{s\ge r} (s^{-(\eta + \beta)}e^{V(s)}).
\end{equation}

By Theorem \ref{thm:ess-f} together with \eqref{eq:ou-m1} and \eqref{eq:ou-m2},  
we have 
\[
\lambda_e \lesssim \limsup_{r\rightarrow\infty} \left( \nu_r r^{2-\beta} e^{V(r)}+r^{-(\eta + \beta)} e^{V(r)} \right). 
\]
We assume in addition that $\limsup_{r\rightarrow\infty}r^{-(\eta + \beta)}e^{V(r)}<\infty$ and $C_2<C_1(\eta+\beta)\beta^{-1}$. 
Then for all sufficiently large $R>1$, 
\[
\mu_V(K_o(R)^c)=\int_{K_o(R)^c}e^{-V(d_0(x))}\,m({\rm d}x) 
\gtrsim \int_{K_o(R)^c}d_0(x)^{-(\eta + \beta)}\,m({\rm d}x)\asymp R^{-\beta}.
\]
This yields $\nu_r=0$ and so 
\begin{equation}\label{eq:ou-upper}
\lambda_e \lesssim \limsup_{r\rightarrow\infty} \left(r^{-(\eta + \beta)} e^{V(r)} \right). 
\end{equation}
In particular, if $\lim_{r\rightarrow\infty}r^{-(\eta + \beta)}e^{V(r)}=0$, 
then $\lambda_e=0$.

\begin{rem}\label{rem:ou-type} \rm 
Assume that $E={\mathbb R}^d$. 
If $m$ is the $d$-dimensional Lebesgue measure, 
and if  $J(x,y)\asymp |x-y|^{-(d+\alpha)} \ ((x,y)\in {\mathbb R}^d\times {\mathbb R}^d \setminus {\rm diag})$ for some $\alpha\in (0,2)$, 
then \eqref{eq:ou-volume} and \eqref{eq:ou-jump} are satisfied with $\eta=d$ and $\beta=\alpha$.  
Under this setting, we see by \cite[Corollaries 1.2 and 1.3]{WW15} 
that if $\limsup_{r\rightarrow\infty}r^{-(d + \beta)}e^{V(r)}=\infty$, 
then $\lambda_e=\infty$.
Hence \eqref{eq:ou-upper} provides an effective upper bound 
for the bottom of the essential spectrum for $({\cal E},{\cal F})$.
\end{rem}

\noindent
{\bf Acknowledgements.} 
The author would like to thank Professor Masayoshi Takeda and Professor Toshihiro Uemura 
for their valuable discussion on the topic in this paper. 
He is grateful to Professor Tao Wang for sending him the reference \cite{HW25}.

\end{document}